\documentclass[a4paper,12pt,oneside,reqno]{amsart}
\usepackage{hyperref}
\usepackage[headinclude,DIV13]{typearea}
\areaset{15.1cm}{25.0cm}
\usepackage{txfonts,amssymb,amsmath,amsthm,bbm}

\usepackage[latin1] {inputenc}
\usepackage{graphicx, psfrag}
\usepackage{xcolor}
\usepackage{subfigure}

\newtheorem{theorem}{Theorem}[section]
\newtheorem{lemma}[theorem]{Lemma}
\newtheorem{proposition}[theorem]{Proposition}

\definecolor{bbm}{RGB}{51,153,0}
\definecolor{above}{RGB}{128,0,128}
\definecolor{below}{RGB}{102,0,204}
\definecolor{cascade}{RGB}{204,0,0}
\definecolor{iid}{RGB}{153,51,0}

\theoremstyle{remark}
\newtheorem*{remark}{Remark}

\def\paragraph#1{\noindent \textbf{#1}}

\numberwithin{equation}{section}

\def\d{\mathrm{d}}

\def\<{\langle}
\def\>{\rangle}

\def\a{\alpha}
\def\b{\beta}
\def\e{\epsilon}
\def\ve{\varepsilon}

\def\g{\gamma}

\def\s{\sigma}

\def\G{\Gamma}

\def\R{{\Bbb R}}  
\def\N{{\Bbb N}}  
\def\P{{\Bbb P}}  
\def\Z{{\Bbb Z}}

\def\E{{\Bbb E}}

\let\cal=\mathcal

\def\DD{{\cal D}}

\def\HH{{\cal H}}

\def\NN{{\cal N}}

\def\ZZ{{\cal Z}}

 \def \G {{\Gamma}}

 \def \b {{\beta}}
 \def \s {{\sigma}}

 \def \g {{\gamma}}
 
 \def \d {{\delta}}
 \def \a {{\alpha}}

 \def \ba {\begin{array}}
 \def \ea {\end{array}}

 \newcommand{\be}{\begin{equation}}
 \newcommand{\ee}{\end{equation}}

\newcommand{\bea}{\begin{eqnarray}}
 \newcommand{\eea}{\end{eqnarray}}
\def\TH(#1){\label{#1}}\def\thv(#1){\ref{#1}}
\def\Eq(#1){\label{#1}}\def\eqv(#1){(\ref{#1})}

\def\cov{\hbox{\rm Cov}}

\def\sfrac#1#2{{\textstyle{#1\over #2}}}

 \def \1{\mathbbm{1}}
\def\wt {\widetilde}

\def\eee{{\mathrm e}}

\def\as{\hbox{\rm {a.s.}}}

    \def\limlaw{\buildrel \DD\over\rightarrow}

\begin{document}

\title[Fluctuations in p-spin SK models] {Fluctuations of the free energy in p-spin SK models on two scales} 
\author[A. Bovier]{Anton Bovier}
 \address{A. Bovier\\Institut f\"ur Angewandte Mathematik\\
Rheinische Friedrich-Wilhelms-Universität\\ Endenicher Allee 60\\ 53115 Bonn, Germany }
\email{bovier@uni-bonn.de}

\author[A. Schertzer]{Adrien Schertzer}            
\address{A. Schertzer \\ Institut f\"ur Angewandte Mathematik\\
Rheinische Friedrich-Wilhelms-Universität\\ Endenicher Allee 60\\ 53115 Bonn, Germany }

\email{aschertz@uni-bonn.de}

\date{\today}
\begin{abstract} 
20 years ago, Bovier, Kurkova, and L\"owe \cite{BKL} proved a central limit theorem  (CLT) for the fluctuations of the free energy in the $p$-spin version of the Sherrington-Kirkpatrick
model of spin glasses at high temperatures. In this paper we improve their results in two ways. First, we extend the range of temperatures to cover the entire regime where 
the quenched and annealed free energies are known to coincide. Second, we identify the main 
source of the fluctuations as a purely coupling dependent term, and we show 
a further CLT for the deviation of the free energy around this random object. 
\end{abstract}


\thanks{
This work was partly funded by the Deutsche Forschungsgemeinschaft (DFG, German Research Foundation) under Germany's Excellence Strategy - GZ 2047/1, Projekt-ID 390685813 and GZ 2151 - Project-ID 390873048,
through Project-ID 211504053 - SFB 1060, and contract number 2337/1-1.
AS thanks Nicola Kistler for very helpful conversations.
 }

 \maketitle

\section{ \bf Introduction}
The  $p$-spin interaction version of the Sherrington-Kirkpatrick \cite{SK}
is a spin system defined on the hypercube $S_N\equiv \{-1,+1\}^N$ where the random Hamiltonian is 
given in terms of a Gaussian process 
 $X_{\cdot}: S_N \to \R$ given by 
\be\label{def1}
X_\s= \binom{N}{p}^{-\frac{1}{2}} \sum_{1 \leq i_1<i_2,\ldots<i_p \leq N} J_{i_1,i_2,\dots,i_p} \s_{i_1}
\s_{i_2}\cdots
\s_{i_p},
\ee
where the $\{J_{i_1,\dots,i_p}\}_{i_1,\dots, i_p=1}^\infty$  
is a family of independent, standard normal random variables defined on some   probability space $(\Omega, \mathcal F, \P)$. 
Alternatively, $X$ is characterised uniquely as the Gaussian field on $S_N$ withe mean 
zero and covariance 
\be
\E \left(X_\s X_{\s'}\right) \equiv f_{p,N} \left(R_N(\s,\s')\right),
\ee
where 
\be
R_N(\s,\s')\equiv  \frac{1}{N}(\s,\s')\equiv \frac{1}{N}\sum_{i=1}^N \s_i\s'_i
\ee
is the \emph{ overlap} between the configurations $\s, \s'$, and $f_{p,N}$ is of the form (see \cite{BKL})
\be \label{fondamental_zero}
f_{p,N}\left(x\right)=\sum_{k=0}^{[p/2]} d_{p-2k}N^{-k}x^{p-2k}(1+O(1/N)),
\ee
where 
\be
d_{p-2k} \equiv (-1)^k \binom{p}{2k}k!!.
\ee
In particular,  
\be\label{fondamental}
 f_{p,N} \left(x\right)= x^p \left(1+O(1/N) \right)\,, \quad \text{uniformly for} \; x\in [-1, 1]. 
\ee
The model with $p=2$ is the classical SK model, introduced in \cite{SK}, and the general version with $p>3$, by Gardner \cite{G85}. 
 The \emph{ Hamiltonian} is given by 
\be
H_N(\s) \equiv -\sqrt{N} X_\s,
\ee
and the \emph{ partition function} is 
\be
Z_{N}(\beta) \equiv \E_\s\left[  \eee^{-\b H_N(\s)}\right] \equiv 2^{-N}\sum_{\s\in S_N}
 \eee^{ \b\sqrt N  X_\s}.
\ee
Finally, minus  the \emph{ free energy} is 
\be
F_{N}(\b) \equiv \frac 1{N} \ln Z_{N}(\b).
\ee
For $m \in (0,1)$, let
\be
\phi(m)\equiv \frac{1-m}{2}\ln (1-m) +\frac{1+m}{2}\ln (1+m), 
\ee
and  
\be
\beta^{2}_p\equiv \inf_{0<m<1}(1+m^{-p})\phi(m),
\ee
for $p\geq 3$, and $\b_2=1$.
It is a well-known consequence of Gaussian concentration of measure theorems, that the 
free energy is self-averaging in the sense that 
\be
\lim_{N\uparrow\infty} F_N(\b)= \lim_{N\uparrow\infty} \E\left[F_N(\b)\right], \as.
\ee
The existence of the limit on the right-hand side was established in a celebrated paper
by Guerra and Toninelli \cite{GueTo}. 
For $\b<\b_p$, it is even true that the so-called  \emph{quenched free energy} on the right-hand side is equal to the so-called \emph{annealed free energy}, that is
\be
\lim_{N\uparrow\infty} \E\left[F_N(\b)\right]
=\lim_{N \to \infty} \frac{1}{N} \ln \E [Z_N(\b)] = \frac{\b^2}{2}.
\ee
This fact was first proven for $p=2$ by Aizenman, Lebowitz, and Ruelle \cite{ALR} and a very simple proof was given later by Talagrand \cite{T0}. The proof in the case $p\geq 3$ is also due to Talagrand \cite{T3}.
 Note that
\be
 \lim_{p\uparrow +\infty} \beta_p=\sqrt{2\ln 2},
\ee
which is the well-known critical temperature of the REM \cite{De2}.  It is, however, not known whether  $\b_p$ is the true critical 
value in general.
It is natural to ask about fluctuations around this limit. This was first done by
Comets and Neveu \cite{CN}, who used the martingale central limit theorem (CLT) for 
the free energy in the case $p=2$ for all $\b<1$. 
The case $p\geq 3$ was analysed by Bovier, Kurkova, and L\"owe \cite{BKL}, also using  martingale methods. They established a CLT in a range $\b<\tilde \b_p$, for 
some $\tilde \b_p<\b_p$. 
 Our first result extends this to the entire range $\b<\b_p$. 
 
\begin{theorem} \label{tresgolri} For all $p\geq 3$ and  $\beta< \beta_p$, 
\be \label{clt_spin}
N^{\frac{p}{2}} \left(F_{N}(\beta) -\frac{\beta^2}{2}\right)\limlaw\mathcal{N}\left(0, \frac{\beta^4p!}{2} \right), \hbox{\rm as} \; N\uparrow\infty.
\ee
\end{theorem}

The  proof of Theorem \thv(tresgolri) is very different from that in \cite{BKL}
and in a sense closer to that of Aizenman et al. \cite{ALR} in the case $p=2$. In fact, we show that the limiting Gaussian comes from a very explicit term 
\be\Eq(theJ.1)
J_N(\beta) \equiv \frac 1{2N}E_\s\left(\b^2 H_N(\s)^2\right)=\frac{\beta^2}{2 \binom{N}{p}}  \sum_{1\leq i_1<\ldots<i_p \leq N}  J_{i_1, i_2,\dots,i_p}^2.
\ee
$J_N(\b)$ is a sum of independent square integrable random variables, and hence by the law of large numbers, for all $\b$,
\be \label{lln}
\lim_{N \to \infty} J_N(\beta) = \frac{\b^2}{2}\,, \as,
\ee
and by the central limit theorem,
\be \label{classical_clt}
 N^{\frac{p}{2}} \left( J_N(\beta) - \frac{\b^2}{2} \right) \limlaw \mathcal{N}\left(0, \frac{\beta^4p!}{2} \right), \hbox{\rm as} \,N\uparrow\infty.
\ee
That $J_N$ and the $F_N$ have the same limits is not a coincidence. In fact, we prove 
Theorem \thv(tresgolri)  by proving that 
\be\Eq(goto.0)
\lim_{N\uparrow\infty}
 N^{\frac{p}{2}} \left( F_N(\b) -J_N(\beta) \right) =0,
 \ee
 in probability. This naturally leads to the question whether upon proper rescaling, the 
 quantity $F_N(\b)-J_N(\b)$ converges to a random variable. The positive answer is the 
 main result of this paper and given by the following theorem.

\begin{theorem} \label{fluctuationspin0}
For $p>2$ and for all $\beta< \beta_p$, we have
\be\label{megagolri2}
 A_N(p)   \Big(F_{N}(\beta)-J_N(\beta) \Big)\limlaw \mathcal{N}\left(\mu(\b,p), {\sigma\left(\beta,p\right)}^2\right),
 \ee
 where 
 \begin{itemize}
 \item[(i)] For $p$ even, 
 \be
A_N(p)=   N^{\left(\frac{3p}{4}-\frac{1}{2}\right)},\quad
\mu(\beta,p)=0,
\ee
and
\be
\sigma\left(\beta,p\right)^2=
          \frac{\beta^6}{3}\E\left[\left(\sum_{k=0}^{p/2} d_{p-2k}{X}^{p-2k}\right)^3\right].
\ee
 \item[(ii)] For $p$ odd,
\be
A_N(p)= 
         N^{p-1},\quad \mu(\beta,p)=  \frac{-\beta^4p!}{4},
\ee
and 
\be
\sigma\left(\beta,p\right)^2= 
         \frac{\beta^8}{12}\E\left[ \left(\sum_{k=0}^{[p/2]} d_{p-2k}{X}^{p-2k}\right)^4 \right]-\frac{\beta^8 p!^2}{8}.
 \ee
 \end{itemize}Here $X$ is a standard normal random variable and $d_{p-2k}=(-1)^k \frac{p(p-1)\dots(p-2k+1)}{2^kk!}$.
\end{theorem}

Compared to \eqref{clt_spin}, Theorem \ref{fluctuationspin0}  provides a higher-level resolution of the  limiting picture.  In fact, in the course of the proof we also identify exactly
the terms arising in the expansion of the partition function that converge to the 
Gaussian in \eqv(megagolri2). Thus, one might envision that, once these terms are again 
subtracted, on a smaller scale, there appears yet another limit theorem. This might even continue ad infinitum. To prove such a result appears, however, rather formidable and 
will be left to future research.

It is interesting to compare this picture with the $p=2$ case. In that case, the variance of the limiting Gaussian distribution blows up at the critical temperature, and thus detects the phase transition. For $p>2$, this is not the case for the Gaussian from Theorem 
\thv(tresgolri), nor for  the corrections given by Theorem \ref{fluctuationspin0}.
This is of course completely in line with the predictions by theoretical physics pertaining to the so-called Gardner's transition \cite{G85}.  

Results similar to Theorem \thv(tresgolri) have been obtained for several related models. Chen et al. \cite{CDP17} obtained analogous results to \cite{BKL} for mixed $p$-SK models, i.e. where the Hamiltonian is given 
as a linear combination of terms of type \eqv(def1) with different $p$ where only even $p$ appear, and recently this was extended
to the general case by  Banerjee and Belius \cite{BB21}. For spherical SK-models, related results were obtained by Baik and Lee
\cite{BL16}. We are not aware of any results like Theorem \thv(fluctuationspin0). 
The paper is organised as follows. In the next section, we  present the proof of Theorem \ref{tresgolri}. Many of the results obtained in the course of the proof are re-used in Section 3 where Theorem \ref{fluctuationspin0} is proven. In the appendix we state 
two frequently used facts about Gaussian random variables for quick reference. 


\section{ \bf Proof of Theorem \ref{tresgolri}.}  

In view of \eqv(classical_clt), to prove Theorem \thv(tresgolri), it is enough to establish that \eqv(goto.0) holds for all  $\b < \b_p$. Setting
\be
\mathcal{Z}_N(\b) \equiv Z_{N}(\b)\eee^{-NJ_N(\b)},
\ee
this amounts to showing that, for $\b < \b_p$,
\be\label{Th0}
\lim_{N \to \infty }N^{\frac{p-2}{2}}\ln \mathcal{Z}_N(\b) =0, \ \text{in probability}.
\ee
The proof of \eqref{Th0} turns out to be remarkably difficult if the entire range $\b<\b_p$ 
is to be covered. This will require a truncation.   For $\epsilon>0$, we set
\be\label{defZ}
\mathcal Z_N(\b) = Z_\epsilon^{\leq} + Z_\epsilon^{>}, 
\ee
where
\be
Z_\epsilon^{\leq} \equiv \E_\s\left(\eee^{-\b H_N(\s)} \1_{\{-H_N(\s)\leq (1+\epsilon) \b N\}}\right)\eee^{-N J_N(\b)}, \quad Z_\epsilon^{>} \equiv \E_\s\left(\eee^{-\b H_N(\s)} 
\1_{\{-H_N(\s)> (1+\epsilon) \b N\}}\right)\eee^{-N J_N(\b)}\,,
\ee
where we dropped  obvious dependencies on the parameters $\b, N$ to lighten 
the notation.
 
We decompose
 \be\label{decomposition1}
N^{\frac{p-2}{2}} \ln \mathcal{Z}_N(\b) = N^{\frac{p-2}{2}} \ln\left( \frac{\mathcal{Z}_N(\b) }{{Z_{\epsilon}^{\leq}}} \right) +N^{\frac{p-2}{2}} \ln\left(\frac{{Z_{\epsilon}^{\leq}}}
{\E [Z_{\epsilon}^{\leq}]}\right)+ N^{\frac{p-2}{2}} \ln \E [Z_{\epsilon}^{\leq} ] .
\ee
The assertion of the theorem then follows from the fact that all three terms on the right-hand side of \eqv(decomposition1) converge to zero in probability.

\begin{proposition}\TH(key.1) 
\begin{itemize}
\item[(i)] For any $q \in \N$, $\b < \b_p $ and small enough $\epsilon = \epsilon(\b, p) >0$, 
\be\Eq(K.1)
\lim_{N\uparrow +\infty} N^{q} \ln\left(\frac{\mathcal{Z}_N(\b) }{{Z_{\epsilon}^{\leq}}}\right)=0,\; \text{in probability}.
\ee
\item[(ii)] For $\b < \b_p $ and small enough $\epsilon = \epsilon(\b, p) >0,$
\be\Eq(K.2)
\lim_{N\uparrow +\infty} N^{\frac{p-2}{2}} \ln\left(\frac{{Z_{\epsilon}^{\leq}}}{\E {Z_{\epsilon}^{\leq}}}\right)=0,\;  \text{in probability},
\ee
\item[(iii)] For any $\b,  \epsilon > 0$,
\be\Eq(K.3)
\lim_{N\uparrow +\infty} N^{\frac{p-2}{2}} \ln \E [Z_{\epsilon}^{\leq}] =0\,.
\ee
\end{itemize}
\end{proposition}
\begin{remark} The fact that \eqv(K.1) holds for all $q\in \N$ is not needed here, but will be used in the proof of Theorem \thv(fluctuationspin0).
\end{remark}

The proof of Proposition \thv(key.1) relies on computations of moments that are combinatorially rather complex.

We introduce some convenient notation. First, we denote by  $I_N$ 
the set of all strictly increasing $p$-tupels in $\{1,\dots, N\}$,  
\be \label{ienne}
I_N \equiv\left\{ (i_1, i_2, \dots, i_p) \in \{1,\dots, N\}^p,\  i_1< i_2<\dots< i_p  \right\}.
\ee
For  $A = (i_1, \dots, i_p) \in I_N$ we write 
\be
\s_A \equiv \s_{i_1} \s_{i_2}\cdots \s_{i_p}, \text{and}\quad J_A \equiv J_{i_1, \dots, i_p}.
\ee
We abbreviate
\be
a_N \equiv  \sqrt N \binom{N}{p}^{-1/2}.
\ee
We can thus write
\be \label{fund_repr}
H_N(\s) =- a_N \sum_{A\in I_N} J_A \s_A, \qquad \text{and}\quad J_N(\b) = \frac{\b^2}{2N} a_N^2 \sum_{A \in I_N} J_A^2. 
\ee
For $a_N, b_N \geq 0$ we write $a_N \lesssim b_N$ if $a_N \leq C b_N$ for some numerical constant $C>0$. 

Finally, we will denote by $\mathfrak  c>0$ a numerical constant, not necessarily the same at different occurrences. 

\subsection{First moments of $\mathcal Z_N(\b)$ and $Z_\epsilon^\leq$, and proof 
of part (iii) of Proposition \eqv(key.1)}  \label{katrois}

We will show that 

\begin{lemma}\TH(Zmoment.1) With the notation above, 
\be \label{firstmomz}
\E\left[ \mathcal Z_N(\b)\right] = 1- \frac{\b^4}{4} N a_N^2 +\frac{\b^8}{32} N^2 a_N^4 +O\left(N^{3-2p}\right).
\ee
\end{lemma}
\begin{proof}
Interchanging  the order of integration, we have
\be\label{moment1'}
\E\left[ \mathcal Z_N(\b)\right] = \E \left[\E_\s\left(\eee^{-\b H_N(\s)-NJ_N(\b)} \right) \right]=\E_\s \left(\E\left[\eee^{-\b H_N(\s)-NJ_N(\b)} \right] \right).
 \ee
 Using \eqv(fund_repr) and the independence to the $J_A$, 
\be  \label{moment1''}
\E\left[\eee^{-\b H_N(\s)-NJ_N(\b)} \right] 
= \prod_{A \in I_N} \E \left[\eee^{\b a_N J_A \s_A  -\frac{\b^2}{2} a_N^2 J_A^2}\right].
\ee
 Computing the Gaussian integral, we get
\be\label{moment11'}
\E \left[\eee^{\b a_N J_A \s_A  -\frac{\b^2}{2} a_N^2 J_A^2}\right]
=  \eee^{\left( {\frac{\b^2 a_N^2\s_{A}^2}{2\left(1+\b^2 a_N^2\right)}} \right)} 
\frac{1}{\sqrt{1+\b^2 a_N^2}}.
\ee
Since $\s_{A}^2=1$, 
this implies that 
\be\label{moment12}
\E\left[ \mathcal Z_N(\b) \right]= 
\exp\left({|I_N|\left(  \frac{\b^2 a_N^2}{2\left(1+\b^2 a_N^2\right)} - \frac{1}{2} \ln\left( 1+\b^2 a_N^2\right) \right) }\right).
\ee
Moreover, using that $\left|I_N\right|=\binom{N}{p}$, and by Taylor expansion we obtain
\bea\label{moment13}\nonumber
\E\left[ \mathcal Z_N(\b) \right]&=&\exp\left( \frac{\binom{N}{p}}{2} \left(\b^2 a_N^2- \b^4 a_N^4+O( a_N^6)-\b^2 a_N^2+ \frac{ \b^4 a_N^4}{2}\right)\right)\\
&=& 1- \frac{\b^4}{4} N a_N^2 +\frac{\b^8}{32} N^2 a_N^4 +O\left(N^{3-2p}\right),
\eea
which is  \eqref{firstmomz}. 
\end{proof}
From \eqref{firstmomz} it follows that $\ln \E\left[ \mathcal Z_N(\b)\right] =O(N a_N^2)$,
and $Na_N^2 =O(N^{2-p})$, it follows that 
$  N^{(p-2)/2 }\ln \E\left[ \mathcal Z_N(\b)\right]=O(N^{1-p/2})$, which tends to zero
for $p\geq 3$.  
The next lemma states that $Z_\epsilon^\leq$ and $\mathcal Z_N(\b)$ are exponentially close, which will imply \eqv(K.3),
\begin{lemma}\TH(veryclose.1)  For any $\e>0$, 
\be\label{diffexp}
\E \left[ \left|Z_\epsilon^\leq-\mathcal Z_N(\b)\right|\right] \leq \exp{\left(-\b^2 N \epsilon^2/2
+O(N^{2-p}) \right)}.
\ee
\end{lemma}
\begin{proof}
Since $\mathcal Z_N(\b)-Z_\epsilon^\leq=Z_\epsilon^>$, we just have to control the 
expectation of the latter. 
Interchanging the order of integration, we obtain, using the Hölder inequality,
\bea\label{fm'}\nonumber
\E \left(Z_\epsilon^>\right) &=&\E_\s \left(\E\left(\eee^{-\b H_N(\s)} \1_{\{-H_N(\s)> (1+\epsilon) \b N\}}\eee^{-N J_N(\b)}\right)\right)  \\\nonumber
&\leq& \E_\s \left(\E \left(\eee^{ -q_1 \b H_N(\s)} \1_{\{-H_N(\s)> (1+\epsilon) \b N \}}\right)^{1/ q_1} \E  \left( \eee^{-q_2 N  J_N(\b)} \right)^{1/ q_2}\right)\nonumber\\
&=& \E_\s \left(\E \left(\eee^{ q_1\b \sqrt N X_\s} \1_{\{X_\s> (1+\epsilon) \b \sqrt N \}}\right)^{1/ q_1} \E  \left( \eee^{-q_2 N  J_N(\b)} \right)^{1/ q_2}\right),
\eea
 for  $1/ q_1+1/ q_2=1$.
 Classical Gaussian estimates (see {\bf Fact I} in the Appendix) yield that
\be\label{fm1}
\E \left(\eee^{ q_1\b \sqrt N X_\s} \1_{\{X_\s> (1+\epsilon) \b \sqrt N \}}\right)^{1/ q_1}
\leq \eee^{ -\frac{(1+\epsilon)^2 \b^2 N}{2 q_1}+(1+\epsilon) \b^2 N }.
\ee
Note that this bound is independent of $\s$. 
It  remains to calculate the second term on the r.h.s. of \eqref{fm'}. By independence of the $J$'s, 
\bea\label{fm3}
\E  \left( \eee^{-q_2N J_N(\b)}\right) 
&=& \left[\E  \left( \eee^{ -\frac{q_2\b^2}{2} a_N^2  J_{A}^2}\right)\right]^{\binom{N}{p}}
=\left(1+q_2\b^2a_N^2\right)^{-\frac 12 \binom{N}{p}}
\nonumber\\&=&\exp\left({- \sfrac {\binom{N}{p}}2 \ln\left(1+\sfrac{N\b^2q_2}{\binom{N}{p}}\right)}\right)
=\exp\left({- \sfrac {N\b^2 q_2}2 +O(N^{2-p})}\right).
\eea
Combining \eqv(fm1) and \eqv(fm3), we obtain, for any $q_1>1$,
\bea\label{fm3''}
\E \left(Z_\epsilon^>\right)  &\leq&
  \exp{\left(-\sfrac{(1+\epsilon)^2 \b^2 N}{2 q_1}+(1+\epsilon) \b^2 N -\sfrac{N\b^2}{2} + O\left(N^{2-p}\right)\right)}\nonumber\\
  &=& \exp\left(-\sfrac{\b^2 N}{2q_1}\left(\e^2 +(1-q_1)(1+2\e) -O(N^{1-p}) \right)\right).
  \eea
But this implies the assertion of the lemma.
\end{proof}
Note that Lemma \thv(veryclose.1) and \eqv(moment13) imply that 
\be\label{moment13.1}
\E\left[ Z_\e^\leq\right]
= 1- \frac{\b^4}{4} N a_N^2 +\frac{\b^8}{32} N^2 a_N^4 +O\left(N^{3-2p}\right).
\ee

Combining Lemma \thv(Zmoment.1) and Lemma \thv(veryclose.1) proves 
\eqv(K.3).

\subsection{The second moment of ${Z_\epsilon^{\leq}}$, and proof of part (ii) of 
Proposition \thv(key.1)} \label{kadeux} 

We set 
\be
\Xi_\epsilon \equiv \frac{{Z_{\epsilon}^{\leq}}-
\E {[Z_{\epsilon}^{\leq}]}}{\E [{Z_{\epsilon}^{\leq}}]} \,.
\ee
\eqv(K.2) is then equivalent to the following lemma.
\begin{lemma}
\TH(K.2.1)
For any $\ve>0$ and $\b<\b_p$,
\be \label{k2bis}
\lim_{N\uparrow\infty} \P\left(\left| N^{\frac{p-2}{2}} \ln\left(   1+ \Xi_\epsilon  \right) \right| \geq \ve
 \right) = 0. 
\ee
\end{lemma}
\begin{proof} Using the Chebyshev inequality and the fact that, for $|x|\leq  1/10$,  
$(\eee^x- 1)^2\geq x^2/2$, for $N$ large enough, 
\be\Eq(nownumber.1)
 \P\left(\left| N^{\frac{p-2}{2}} \ln\left(   1+ \Xi_\epsilon  \right) \right| \geq \ve \right) 
\leq  \frac {\E\left[\Xi_\e^2\right]}{\left(\eee^{\ve N^{1-p/2}}-1\right)^2}
+ \frac {\E\left[\Xi_\e^2\right]}{\left(\eee^{-\ve N^{1-p/2}}-1\right)^2}
\leq 8\ve^{-2} N^{p-2}{\E\left[\Xi_\e^2\right]}.
\ee
Since $\E\left[\Xi_\e^2\right]=\frac{\E\left[(Z^\leq_\e)^2\right]-\E\left[Z^\leq_\e\right]^2}
{\left(\E\left[Z^\leq_\e\right]\right)^2}$, and $\E\left[\Z^\leq_\e\right]$ is already computed, we only need to compute a precise bound on the second moment of $Z_\e^\leq$.

We write  $\E_{\s,\s'}=\E_\s \E_{\s'}$ and set $\Gamma_N \equiv \{ -1,-1-\frac{2}{N}, \dots ,1\}$. Then, for any function 
 $G: \R \to \R$, one has  
\bea\label{switch2}\nonumber
\E_{\s,\s'}\left[ G\left({\binom{N}{p}}^{-1}\sum_{i_1<i_2<\dots<i_p}\s_{i_1}\s_{i_1}' \dots \s_{i_p}\s_{i_p}' \right)\right]& =&\E_{\s,\s'}\left[ G\left(\cov(X_\s,X_{\s'})\right)\right]\\
&=&\sum_{m \in \Gamma_N} G[f_{p,N}\left(m\right)] p_N(m),
\eea
where $p_N(m) \equiv \P_{\s,\s'}( R_N(\s,\s')=m)$. 

 With this in mind, we split the second moment
  according to the value of the overlap
\bea\label{mom21'}
\E \left[\left(Z_{\epsilon}^{\leq}\right)^2\right]
&= &\E_{\s,\s'} \E \left(\eee^{\b \sqrt N \left(X_\s+X_{\s'}\right)} \1_{\{X_\s \leq(1+\epsilon) \b \sqrt N\}}  \1_{\{X_\s' \leq(1+\epsilon) \b \sqrt N\}}  \eee^{-2NJ_N(\b)} \right)\nonumber\\
 &=& \E_{\s,\s'} \E \left(\eee^{\b \sqrt N \left(X_\s+X_{\s'}\right)} 
 \1_{\{X_\s \leq(1+\epsilon) \b \sqrt N\}}  \1_{\{X_\s' \leq(1+\epsilon) \b \sqrt N\}} 
\eee^{-2NJ_N(\b)} \1_{\{ \lvert R_N(\s,\s') \rvert<(\ln N)^{-1}\}}  \right)\nonumber\\\nonumber
&&+\E_{\s,\s'} \E \left(\eee^{\b \sqrt N \left(X_\s+X_{\s'}\right)} 
\1_{\{X_\s \leq(1+\epsilon) \b \sqrt N\}}  \1_{\{X_\s' \leq(1+\epsilon) \b \sqrt N\}} 
 \eee^{-2NJ_N(\b)} \1_{\{ \lvert R_N(\s,\s') \rvert \geq (\ln N)^{-1}\}} \right)\\
&&\equiv A+B.
\eea
 We will  now prove that the $B$-term (large overlap) is subexponentially small and compute 
 the leading orders of the $A$-term.

\begin{lemma}\TH(B.1)
For all $\b<\b_p$, there exists $\e_0>0$  and a constant $\mathfrak c$ such that, for all
$0\leq \e<\e_0$, 
\be \label{subsm}
B \leq \exp\left( - \mathfrak c N \ln(N)^{-p} \right).
\ee
\end{lemma}

\begin{lemma}\TH(A.1)
For any $\b$, 
\be \label{polysm}
A =  1-\frac{\b^4N a_N^2}{2}  +O \left( N^{3-3p/2}\right).
\ee
\end{lemma}

\begin{proof}[Proof of Lemma \thv(B.1)]
 To simplify the notation, set
$b_N=(\ln N)^{-1}$ and  $B_N \equiv \{ \lvert R_N(\s,\s') \rvert \geq b_N\} $.
 We simplify the constraints by using that 
 \be
 \1_{\{X_\s \leq(1+\epsilon) \b \sqrt N\}}  \1_{\{X_\s' \leq(1+\epsilon) \b \sqrt N\}} 
\leq\1_{\{X_\s+X_{\s'} \leq 2(1+\epsilon) \b \sqrt N\}}.
\ee
 By H\"older's inequality, we then get 
\be\label{mom2t}
 B \leq \E_{\s,\s'} \bigg(  \E  \left(\eee^{q_1 \b \sqrt N \left(X_\s+X_{\s'}\right)} \1_{\{X_\s+X_\s'\leq(1+\epsilon) 2 \b \sqrt N\}}\right)^{\frac{1}{q_1}} \E \left(\eee^{-2q_2NJ_N(\b)} \right)^{\frac{1}{q_2}}  \1_{B_N} \bigg),
\ee
with $q_1, q_2 \geq 1$ satisfying $1/ q_1+1/ q_2=1$. 
Since $X_\s+X_{\s'}$ is a Gaussian random variable with mean zero and 
variance $2(1+f_{p,N}(R_N(\s,\s'))$, 
the right hand side can be written as 
\be\label{mom3t}
 \E_{\s,\s'} \left( \E  \left(\eee^{q_1 \b \sqrt{N\left(2+2 f_{p,N} \left(R_N(\s,\s')\right)\right)} \xi } \1_{\left\{\xi \leq(1+\epsilon) \sqrt 2 \b \sqrt N {\left(1+ f_{p,N} \left(R_N(\s,\s')\right)\right)^{-1/2}}\right\}}\right)^{\frac{1}{q_1}}   \E\left(\eee^{-2q_2 NJ_N(\b)} \right)^{\frac{1}{q_2}}   \1_{B_N}  \right),
\ee
where $\xi$ is a standard Gaussian. As in \eqref{fm3},
\be\label{mom3t'}
\E\left(\eee^{-2q_2 NJ_N(\b)} \right)^{\frac{1}{q_2}}  \leq \eee^{-\b^2 N+ O\left( N^{2-p}\right)},
\ee
and the first term is bounded as in \eqv(fm1), which yields
\bea
\Eq(asbefore.1)
\nonumber
&& \E  \left(\eee^{q_1 \b \sqrt{N\left(2+2 f_{p,N} \left(R_N(\s,\s')\right)\right)} \xi } \1_{\left\{\xi <(1+\epsilon) \sqrt 2 \b \sqrt N {\left(1+ f_{p,N} \left(R_N(\s,\s')\right)\right)^{-1/2}}\right\}}\right)^{\frac{1}{q_1}} 
 \\
 &&\quad\leq 
 \exp\left({-\frac{(1+\epsilon)^2 \b^2 N}{q_1\left(1+ f_{p,N} \left(R_N(\s,\s')\right)\right)}+(1+\epsilon) 2 \b^2 N}\right).
 \eea
Combining these two steps, we obtain 
\be\label{mom2t1}
 B \leq  \E_{\s,\s'} \bigg(   \1_{B_N}\exp{\left(-\frac{(1+\epsilon)^2 \b^2 N}{q_1\left(1+
  f_{p,N} \left(R_N(\s,\s')\right)\right)}+(1+\epsilon) 2 \b^2 N -\b^2 N+ O\left( { N^{2-p}} \right)\right)}\bigg). 
\ee
Since this holds for all $q_1>1$, one sees that the exponential term in \eqv(mom2t1) is 
bounded by
\be\label{mom4t}
\exp{\left(-\b^2 N \left( \frac{\epsilon^2-(1+2\epsilon)f_{p,N} \left(R_N(\s,\s')\right)}{\left(1+ f_{p,N} \left(R_N(\s,\s')\right)\right)}+ O\left( N^{1-p} \right) \right)\right)}.
\ee
By Stirlings estimate, we have
\be\label{mom5t}
p_N(m)= \binom{N}{N\frac{(1+m)}{2}} 2^{-N} \leq \exp{\left(-N\phi (m)\right)}.
\ee
Using \eqref{switch2} and plugging \eqref{mom4t} and \eqref{mom5t} into \eqref{mom2t1} gives
\be\label{mom6t}
B \lesssim \sum_{\substack{m \in \Gamma_N, \\ \lvert m \rvert \geq b_N}} \exp{ \left( N\left( -\b^2 \frac{\epsilon^2-(1+2\epsilon)f_{p,N} \left(m \right)}{\left(1+ f_{p,N} \left(m\right)\right)} -\phi (m)\right)\right)}.
\ee
 We write
\bea \label{deltaenne}\nonumber
\d_N & \equiv& -\b^2 N \frac{\epsilon^2-(1+2\epsilon)f_{p,N} \left(m \right)}{\left(1+ f_{p,N} \left(m\right)\right)} -\phi (m) N \\\nonumber
& =& -\b^2 N \frac{\epsilon^2}{1+f_{p,N}(m)} 
+ \frac{f_{p,N}(m) N}{1+ f_{p,N}(m)} \left[ \b^2(2\epsilon +1) - \left( 1+ f_{p,N}(m)^{-1}\right) \phi(m) \right] \\
& \leq& \frac{f_{p,N}(m) N}{1+ f_{p,N}(m)}  \left[\b^2(2\epsilon +1)  - \left( 1+ f_{p,N}(m)^{-1}\right) \phi(m) \right].
\eea 
If $p$ is even or $m\geq 0$, 
recalling that $\b_p^2\equiv \inf_{0<m<1}(1+m^{-p})\phi (m)$, the last line in \eqv(deltaenne) is 
\be
\nonumber\\\leq  \frac{f_{p,N}(m) N}{1+ f_{p,N}(m)}  \left[\b^2(2\epsilon +1)  - \b_p^2 \right].
\ee
 Since $f_{p,N}(m) =m^p+O(1/N)$,  for $|m|\geq b_N$, 
\be
\d_N\leq -\frac {N b_N^p+O(1)}{2} (\b_p^2-\b^2(1+2\e)). \label{basta} 
\ee
This gives 
\be
B\lesssim N \eee^{ -\frac {N b_N^p+O(1)}{2} (\b_p^2-\b^2(1+2\e))}.
\ee
If $p$ is odd and  $m<0$, $1+f_{p,N}(m)^{-1}\leq 0$, and 
 we immediately obtain
\be
\d_N \leq -Nb_N^p \b^2,
\ee
which is even better. This proves Lemma \thv(B.1)
\end{proof}

Next we prove Lemma \thv(A.1)

\begin{proof}[Proof of Lemma \thv(A.1)] 
We have to decompose the term $A$ further according to the value of the overlap. 
For $\a$ satisfying 
\be \label{alphaa}
\frac{1}{p}<\alpha <\frac{1}{2},
\ee 
we set $A=A_1+A_2$, where
\be \label{mom22}
A_1  \equiv \E_{\s,\s'} \E \left(\eee^{\b \sqrt N \left(X_\s+X_{\s'}\right)} \1_{\{X_\s \leq(1+\epsilon)  \b \sqrt N\}} \1_{\{X_{\s'} \leq(1+\epsilon)  \b \sqrt N\}} \eee^{-2NJ_N(\b)} \1_{N^{-\a}\leq |R_N(\s,\s')|<b_N} 
 \right),
 \ee
 and 
 \be\Eq(mom22.1)
A_2 \equiv \E_{\s,\s'} \left( \E \left(\eee^{\b \sqrt N \left(X_\s+X_{\s'}\right)} 
\1_{\{X_\s \leq(1+\epsilon)  \b \sqrt N\}} \1_{\{X_{\s'} \leq(1+\epsilon)  \b \sqrt N\}}\eee^{-2NJ_N(\b)}  \right) \1_{|R_N(\s,\s')|<N^{-\a}}   \right).
\ee
The point is that  $A_1$ is very small, even if we drop the constraints on $X_\s$ and $X_{\s'}$, whereas $A_2$ has to be computed precisely. 

Thus, we bound $A_1$ by
\be
\Eq(plop.1)
0\leq A_1\leq\E_{\s,\s'} \E \left(\eee^{\b \sqrt N \left(X_\s+X_{\s'}\right)}  \eee^{-2NJ_N(\b)} \1_{\{N^{-\a}\leq |R_N(\s,\s')|<b_N\}}\right) .
\ee
 Using the independence of the Gaussian variables
 \be \label{mom23}
\E \left(\eee^{\b \sqrt N \left(X_\s+X_{\s'}\right)} \eee^{-2NJ_N(\b)}  \right)=\prod_{K \in I_N} \E \left(\eee^{ \b a_N  J_{K}\left(\s_{K}+\s_{K}'\right) -\b^2 a_N^2  J_{K}^2} \right).
 \ee
 Computing the Gaussian integrals,
 \be
  \E \left(\eee^{ \b a_N  J_{K}\left(\s_{K}+\s_{K}'\right) -\b^2 a_N^2  J_{K}^2} \right)
 =\eee^{\left(1+\s_{K} \s_{K}'\right) \left(\frac{\b^2a_N^2}{2\b^2 a_N^2+1}\right)-\frac{\ln(1+2 \b^2 a_N^2)}{2}},
 \ee
 and so 
 \be \label{mom24'}
\E \left(\eee^{\b \sqrt N \left(X_\s+X_{\s'}\right)} \eee^{-2NJ_N(\b)}  \right)
= \eee^{\left( \sum_{K \in I_N} \s_{K}\s_{K}'\right) \left(\frac{\b^2a_N^2}{2 \b^2 a_N^2+1}\right)}\eee^{\binom{N}{p} \left(\frac{\b^2 a_N^2}{2\b^2 a_N^2+1} -\frac{\ln(1+2 \b^2 a_N^2)}{2}\right)}.
\ee
As in \eqv(fm3), we have
\be \label{mom25}
\exp\left({\binom{N}{p} \left(\frac{\b^2 a_N^2}{2\b^2 a_N^2+1} -\frac{\ln(1+2 \b^2 a_N^2)}{2}\right)}\right)=\exp{\left( - \b^4 N a_N^2+O\left( {N^{3-2p}}\right)\right)}.
\ee
Thus 
\be\label{facile}
A_1 \leq \sum_{\substack{m \in \Gamma_N \\ N^{-\alpha} \leq \lvert m \rvert \leq b_N}}  \exp\left( \frac{\b^2 N f_{p,N} \left( m \right)} {2\b^2 a_N^2+1} \right) p_N(m)
\leq \sum_{\substack{m \in \Gamma_N \\ N^{-\alpha} \leq \lvert m \rvert \leq b_N}}  \exp\left(N\left( \sfrac{\b^2 f_{p,N} \left( m \right)} {2\b^2 a_N^2+1}-\sfrac{m^2}2\right) \right),
\ee
where the last inequality uses \eqref{mom5t}.  
Using the asymptotics for $f_{p,N}$,  we get that 
\be\label{facile2'}
A_1 \leq \sum_{\substack{m \in \Gamma_N \\ N^{-\alpha} \leq \lvert m \rvert \leq b_N}}  \exp { \left( N \frac{m^2}{2}  \left(\frac{2 \b^2 m^{p-2}\left(1+o_N(1)\right)}{2 \b^2 a_N^2+1}-1 \right)\right)}.
\ee
In the range of summation, 
\be
\left|\frac{2 \b^2 m^{p-2}\left(1+o_N(1)\right)}{2 \b^2 a_N^2+1}\right|\lesssim b_N^{p-2},
\ee 
which tends to zero,
and  thus, using also the lower bound on $|m|$,  
\be\label{facile2}
A_1  \lesssim  N \exp\left(-N^{1-2\alpha}/2\right) .
\ee

For $A_2$, the constraints on the $X_\s,X_{\s'}$ can also be dropped, but this is more subtle. We write $A_2=A_{21}+R_2$, where 
\be\Eq(mom22.2)
A_{21} \equiv \E_{\s,\s'} \left( \E \left(\eee^{\b \sqrt N \left(X_\s+X_{\s'}\right)} 
\eee^{-2NJ_N(\b)}  \right) \1_{|R_N(\s,\s')|<N^{-\a}}   \right).
\ee
We first compute $A_{21}$.
We set
$\Gamma_N^{\alpha} \equiv \{m \in \Gamma_N, \lvert m \rvert \leq N^{-\alpha}\}$.
Using \eqref{switch2}, we have
\be \label{mom24}
A_{21}\exp{\left( + \b^4 N a_N^2+O\left( {N^{3-2p}}\right)\right)}\equiv \wt A_{21}= \sum_{m \in \Gamma^\a_N}  \exp{\left( \frac{\b^2 N f_{p,N} \left( m \right) } {2\b^2 a_N^2+1}\right)} p_N\left(m\right).
\ee
To deal with this term, we use the following standard bound for the exponential,
\be\Eq(lessstrange)
\left|\exp( \xi) - 1 - \xi  - \frac{1}{2}\xi^2-\frac{1}{3!}\xi^3 \right|\leq \frac{1}{4!}\xi^4 \exp |\xi|,
\ee
with $\xi=\frac{\b^2 N f_{p,N} \left( m \right) } {2\b^2 a_N^2+1}$.
Notice that on $\G^\a_N$, $N f_{p,N}(m)\leq  N^{1-p\a}$, which tends to zero, as $N\uparrow \infty$. 
Hence, on the domain of summation of \eqv(mom24), $\exp(|\xi|)\leq \eee^\mathfrak c$. 
This allows us to bound $A_{21}$ as 
\be
\Eq(strange)
\left|\wt A_{21}-\sum_{m \in \Gamma^\a_N}  \left(1+\xi  + \frac{1}{2}\xi^2+\frac{1}{3!}\xi^3 \right
)p_N(m)\right|
\leq \frac{1}{4!}\sum_{m \in \Gamma^\a_N}\xi^4 \eee^{\mathfrak c}p_N(m).
\ee
Moreover, the sum over the terms on the left-hand side can be extended to sums over all
 of $\G_N$ with just an exponentially small error. 
\be
\Eq(strange.2)
\left|\wt A_{21}-\sum_{m \in \Gamma_N}  \left(1+\xi  + \frac{1}{2}\xi^2+\frac{1}{3!}\xi^3 \right
)p_N(m)\right|
\leq \frac{1}{4!}\sum_{m \in \Gamma_N}\xi^4 \eee^{\mathfrak c}p_N(m) +O\left(\eee^{-N^{1-2\a}}\right).
\ee

The  sums over the $\xi^k$ can be computed fairly well by re-expressing them in terms of expectations over the $\s$. 
Namely
\be\Eq(goodbound.1)
 \sum_{m \in \Gamma_N} f_{p,N}(m)p_N(m)=\binom{N}{p}^{-1}\E_{\s,\s'}\left( \sum_{A\in I_N} 
 \s_A\s'_A\right) =0,
 \ee
 \be\Eq(goodbound.2)
 \sum_{m \in \Gamma_N} f_{p,N}(m)^2p_N(m)=\binom{N}{p}^{-2}\E_{\s,\s'}\left( \sum_{A\in I_N} 
 \s_A\s'_A\right)^2 =\binom{N}{p}^{-1}, 
 \ee
 and, for $k\geq 3$
 \be\Eq(badbound.2)
 \sum_{m \in \Gamma_N} f_{p,N}(m)^kp_N(m)=\binom{N}{p}^{-k}\E_{\s,\s'}\left( \sum_{A\in I_N} 
 \s_A\s'_A\right)^k \leq \binom{N}{p}^{-k}N^{pk/2},
 \ee
 since all indices must occur alt least twice.
From this we obtain 
\be\Eq(badbound.1)
\wt A_{21}= 1+ \frac{\b^4Na_N^2}{2(1+2\b^2a_N^2)^2} +O\left(N^{3(1-p/2)}\right)
=1+ \frac{\b^4Na_N^2}{2} +O\left(N^{3(1-p/2)}\right).
\ee
Finally, we bound $R_2$. Note that 
\be\Eq(mom22.3)
|R_2|\leq 2 \E_{\s,\s'} \left( \E \left(\eee^{\b \sqrt N \left(X_\s+X_{\s'}\right)} 
\1_{\{X_\s >(1+\epsilon)  \b \sqrt N\}} \eee^{-2NJ_N(\b)}  \right) \1_{|R_N(\s,\s')|<N^{-\a}}   \right).
\ee
The idea here is that under the constraint on $R_N(\s,\s')$, $X_\s$ and $X_{\s'}$ 
are almost independent. Using Hölder's inequality as before,
\bea\nonumber
 \E \left(\eee^{\b \sqrt N \left(X_\s+X_{\s'}\right)} 
\1_{\{X_\s >(1+\epsilon)  \b \sqrt N\}} \eee^{-2NJ_N(\b)}  \right)
&\leq& \left(\E \left(\eee^{q_1\b \sqrt N \left(X_\s+X_{\s'}\right)}
\1_{\{X_\s >(1+\epsilon)  \b \sqrt N\}} \right)\right)^{\frac 1{q_1}}
\\
&&\quad\times\left(\E \left( \eee^{-2q_2NJ_N(\b)}  \right)\right)^{\frac 1{q_2}}.
\eea
As in \eqv(fm3),  we get for the second factor
\be
\left(\E \left( \eee^{-2q_2NJ_N(\b)}  \right)\right)^{\frac 1{q_2}}
\leq \eee^{-N\b^2 +O(N^{2-p})}.
\ee
To deal with with first factor, we 
notice that $X_{\s'}$ can be written as
\be
X_{\s'}= \g X_\s +\sqrt {1-\g^2} \xi,
\ee
where $\xi $ is a normal random variable 
 independent of $X_\s$ and $\g = f_N(R_N(\s,\s'))$. 
Hence
\be
\E \left(\eee^{q_1\b \sqrt N \left(X_\s+X_{\s'}\right)} 
\1_{\{X_\s >(1+\epsilon)  \b \sqrt N\}} \right)=
\E \left(\eee^{q_1\b \sqrt N X_\s(1+\g)} 
\1_{\{X_\s >(1+\epsilon)  \b \sqrt N\}} \right)
\E \left(\eee^{q_1\b \sqrt N \sqrt{1-\g^2}\xi}\right).
\ee
Using again {\bf Fact 1} and 
since $|R_N(\s,\s')|\leq N^{-\a}$, and that  these bounds hold for all $q_1>1$, 
it follows that
\be
|R_2|\leq   \eee^{-\b^2 N \e^2/2 + o(N)}.
\ee 
With these bounds on $A_1$ and $A_2$, and the bound \eqv(mom25), 
\bea\nonumber
A&=&\left(1+ \frac{\b^4Na_N^2}{2(1+2\b^2a_N^2)^2} +O\left(N^{3(1-p/2)}\right)\right)
\exp{\left( - \b^4 N a_N^2+O\left( {N^{3-2p}}\right)\right)}\\
&=&1-\frac{\b^4Na_N^2}{2} +O\left(N^{3(1-p/2)}\right).
\eea
This implies  \eqv(polysm)
 and concludes the proof of Lemma \thv(A.1).
\end{proof}
We now conclude the proof of Proposition \thv(key.1). Combining  \eqref{subsm} and \eqref{polysm} yields
\be\label{finito4'}
\E\left[(Z_{\epsilon}^{\leq})^2\right] =  1-\frac{\b^4N a_N^2}{2}+O\left(N^{3-3p/2}\right).
\ee
Furthermore, using \eqref{moment13.1} we have that
\be\label{finito5}
\left(\E(Z_{\epsilon}^{\leq})\right)^2=\left(1- \frac{\b^4}{4} N a_N^2 +O\left(N^{4-2p}\right)\right)^2=1- \frac{\b^4 N a_N^2}{2}+O\left(N^{4-2p}\right),
\ee
hence combining \eqref{finito4'} and \eqref{finito5} leads to 
\be\label{finito6}
 \E\left( \Xi_\epsilon^2\right)=\frac{\E ({Z_{\epsilon}^{\leq}}^2)-
 \E ({Z_{\epsilon}^{\leq}})^2}{\E ({Z_{\epsilon}^{\leq}})^2}=\frac{O\left(N^{3-3p/2}\right)}
 {\E ({Z_{\epsilon}^{\leq}})^2} .
\ee
Inserting this into \eqv(nownumber.1), we get
\be
 \P\left(\left| N^{\frac{p-2}{2}} \ln\left(   1+ \Xi_\epsilon  \right)\right|>\ve\right) \leq 8\ve^{-2} O\left(N^{1-p/2}\right),
 \ee
which proves Lemma \thv(K.2.1). 
\end{proof} 
This also concludes the proof of part (ii) of Proposition \thv(key.1).

\subsection{Exponential concentration: proof of  (i) of Proposition \thv(key.1)} \label{kaun}  
Since
\be
N^{q} \ln\left(\frac{\mathcal{Z}_N(\b) }{{Z_{\epsilon}^{\leq}}}\right)=N^{q} \ln\left(\frac{\mathcal{Z}_N(\b) }{\mathcal{Z}_N(\b) -Z_{\epsilon}^{>}}\right)=-N^{q} \ln\left(1-\frac{Z_{\epsilon}^{>}}{\mathcal{Z}_N(\b) }\right),
\ee
the assertion \thv(K.1) in Lemma \thv(key.1) follows from the following lemma.
\begin{lemma}\TH(conc.1) Assume that $\b<\b_p$. Then, 
 For all  $ \ve>0$ there exists $\mathfrak c > 0$ such that 
\be \label{exfa}
\P\left( \frac{Z_{\epsilon}^{>}}{\mathcal{Z}_N(\b) } \geq \ve  \right) \leq \exp( -\mathfrak c N).
\ee
\end{lemma}
\begin{proof}

\bea\label{Markov}
\P\left( \frac{Z_{\epsilon}^{>}}{\mathcal{Z}_N(\b) }\geq \ve \right) &=&\P\left(\frac{\E_\s\left(\eee^{-\b H_N(\s)} \1_{\{-H_N(\s)> (1+\epsilon) \b N\}}\right)}{\E_\s\left(\eee^{-\b H_N(\s)} \right)}  \geq \ve \right)\nonumber\\
& \leq &\frac{1}{\ve} \E \left(\frac{\E_\s\left(\eee^{-\b H_N(\s)} \1_{\{-H_N(\s)> (1+\epsilon) \b N\}}\right)}{\E_\s\left(\eee^{-\b H_N(\s)} \right)} \right).
\eea 

\noindent By Gaussian concentration of measure, it follows that 
\be\label{conc}
\P\left(\left|\ln \E_{\s}\eee^{-\b H_N(\s)}-\E \left( \ln \E_{\s}\eee^{-\b H_N(\s)}\right)\right|
>N\b^2 \frac{\epsilon^2}{4} \right)\leq \exp\left(-N\b^2 \frac{\epsilon^4}{32}\right).
\ee
(See e.g. \cite[(2.56)]{BKL}). 

We introduce the events 
\be
O_{N, \b , \epsilon} \equiv \left \{ \left|\ln \E_{\s}\eee^{-\b H_N(\s)}-\E\left(\ln\E_{\s}\eee^{-\b H_N(\s)}\right)\right|>N\b^2\sfrac{\epsilon^2}{4}\right\}\,,
\ee
and split the r.h.s. of \eqref{Markov} as 
 \bea\label{tricks1'}
 &&\E \left(\frac{\E_\s\left(\eee^{-\b H_N(\s)} \1_{\{-H_N(\s)> (1+\epsilon) \b N\}}\right)}{\E_\s\left(\eee^{-\b H_N(\s)} \right)} \right)\\
\nonumber &&  \leq  \E \left( \1_{ O_{N, \b , \epsilon}^c} \frac{\E_\s\left(\eee^{-\b H_N(\s)} \1_{\{-H_N(\s)> (1+\epsilon) \b N\}}\right)}{\E_\s\left(\eee^{-\b H_N(\s)} \right)} \right)+\P(O_{N, \b , \epsilon}) \\
\nonumber&&\leq \, \E \left( \1_{ O_{N, \b , \epsilon}^c } \frac{\E_\s\left(\eee^{-\b H_N(\s)} \1_{\{-H_N(\s)> (1+\epsilon) \b N\}}\right)}{\E_\s\left(\eee^{-\b H_N(\s)} \right)} \right)+ \exp\left({-N\b^2 \frac{\epsilon^4}{32}}\right) \, ,
\eea 
where  for the first inequality we use that the quotient of the $\E_\s$-terms is smaller than one, and \eqref{conc} is used in the last step. 
On the event $O_{N, \b , \epsilon}^c$ , we have that 
\bea\Eq(zero.101)
\E_\s\left(\eee^{-\b H_N(\s)}\right) &=& \exp\left(\ln \E_\s\left(\eee^{-\b H_N(\s)}\right) -\E\left(\ln \E_\s\left(\eee^{-\b H_N(\s)}\right)\right)
+\E\left(\ln \E_\s\left(\eee^{-\b H_N(\s)}\right)\right)\right)\nonumber\\
&\geq&
\exp\left(\E\left(\ln \E_\s\left(\eee^{-\b H_N(\s)}\right)\right)-N\b^2\e^2/4\right).
\eea
Using this inequality
\be\label{tricks}
\E \left( \1_{\{\{O^{N}_{\b,\epsilon}\}^C\}} \frac{\E_\s\left(\eee^{-\b H_N(\s)} \1_{\{-H_N(\s)> (1+\epsilon) \b N\}}\right)}{\E_\s\left(\eee^{-\b H_N(\s)} \right)} \right)\\
 \leq \eee^{N\b^2 \frac{\epsilon^2}{4}} \frac{ \E\left( \E_{\s}\left(\eee^{-\b H_N(\s)-N\frac{\b^2}{2}}\1_{\{-H_N(\s)>N\b(1+\epsilon)\}}\right) \right)}{\exp\left(
\E\ln\E_{\s}\eee^{-\b H_N(\s)-N\frac{\b^2}{2}}\right)}.
\ee
By classical Gaussian estimates ({\bf Fact I} in Appendix), the numerator on the r.h.s. above reads
\be\label{d1'}
\E\left(\E_\s\left(\eee^{-\b H_N(\s)-N\frac{\b^2}{2} } \1_{\{-H_N(\s)> (1+\epsilon) \b N\}}\right)\right)\leq  \exp\left({-N \b^2 \frac{\epsilon^2}{2}}\right).
\ee
Combining \eqref{tricks1'}, \eqref{tricks} and \eqref{d1'}, we obtain 
\be\label{tricksf}
 \E \left(\frac{\E_\s\left(\eee^{-\b H_N(\s)} \1_{\{-H_N(\s)> (1+\epsilon) \b N\}}\right)}{\E_\s\left(\eee^{-\b H_N(\s)} \right)} \right) \leq \frac{\exp\left({N\b^2 \frac{\epsilon^2}{4}}\right)\exp\left(-{N\b^2 \frac{\epsilon^2}{2}}\right)}{\exp\left(
\E\ln\E_{\s}\eee^{-\b H_N(\s)-N\frac{\b^2}{2}}\right)}+ \exp\left({-N\b^2 \frac{\epsilon^4}{32}}\right).
\ee   
It remains to bound the denominator. 
Note that 
\be \Eq(difference.1)
\E\ln\E_{\s}\eee^{-\b H_N(\s)-N\frac{\b^2}{2}}=\E\ln\E_{\s}\eee^{-\b H_N(\s)}-\ln \E \E_{\s}\eee^{-\b H_N(\s)},
\ee
so this is just the difference between the quenched and annealed free energy.  In the course of the proof that these are asymptotically equal for $\b<\b_p$, 
it it actually shown that
for any $\b< \b_p$, there exists $K>0$ such that
\be \label{d1bis}
-K\sqrt{N}<\E\ln\E_{\s}\eee^{-\b H_N(\s)}-\ln \E \E_{\s}\eee^{-\b H_N(\s)}\leq 0.
\ee
(see e.g. Section 11.2 in \cite{BovStatMech}).
Inserting this estimate into  \eqref{tricksf}, it follows that 
\be\label{tricksff}
 \E \left(\frac{\E_\s\left(\eee^{-\b H_N(\s)} \1_{\{-H_N(\s)> (1+\epsilon) \b N\}}\right)}{\E_\s\left(\eee^{-\b H_N(\s)} \right)} \right) \leq \exp\left(-{N\b^2 \frac{\epsilon^2}{4}+K\sqrt{N}}\right)+ \exp\left({-N\b^2 \frac{\epsilon^4}{32}}\right).
\ee   
This together with the Markov inequality implies \eqv(exfa) and ends the proof of the lemma.
\end{proof}
Thus the proof of Lemma \thv(key.1) is complete and this also concludes the proof of Theorem \thv(tresgolri).
%

\section{\bf Proof of Theorem \ref{fluctuationspin0}.} 
The quantity we need to control can be expressed as 
\be\label{interet}
F_{N}(\b)-J_N(\b)  =\frac{1}{N} \ln\left({{\mathcal{Z}_N(\b)}}\right). 
\ee
The proof of Theorem \thv(fluctuationspin0) relies essentially  on a Taylor expansion of the exponential function  in $\ZZ_N(\b)$. 
 Recalling the definition of $J_N(\b)$, see \eqv(theJ.1),
\be\label{t2'}
\mathcal{Z}_N(\b)=\E_\s\left(\eee^{-\b H_N(\sigma)-{\E_\s\left(\b^2 H_N(\sigma)^2\right)}/{2}}\right).
\ee
Expanding the exponential and  ordering terms in powers of $\b$, we see that  
\be
\mathcal{Z}_N(\b)= T_N(\b)+O_N(\b^5),
\ee
where 
\be \label{magique_t}
T_N(\b) \equiv 1-\b^4\frac{\left(\E_{\sigma}H_N(\sigma)^2\right)^2}{8}-\b^3\frac{\E_{\sigma} \left( H_N(\sigma)^3\right)}{3!}+\b^4\frac{\E_{\sigma} \left( H_N(\sigma)^4\right)}{4!}\,.
\ee
 Writing 
 \be \Eq(ln.1)
 \a_N(p) \ln \ZZ_N(\b) =\a_N(p) \ln\left(1+\ZZ_N(\b)-1\right),
 \ee
 with $\a_N(p)=A_N(p)/N$
 we see that the assertion of the theorem is equivalent to 
 \be
 \Eq(ln.2)
 \a_N(p) \left(\ZZ_N(\b)-1\right)\limlaw\NN\left(\mu(\b,p),\s(\b,p)^2\right).
 \ee
  
The proof of Theorem \thv(fluctuationspin0) will therefore follow  from the following two  lemmata.

\begin{proposition} \label{normal1}
With the notation above, for $p>2$  for \emph{any} $\b>0$, 
\be\label{normalfrere}
\alpha_N(p) \left(T_N(\b)-1\right)\limlaw  \mathcal{N}\left(\mu(\b,p), \sigma(\b,p)^2\right),
 \ee
 as $N\uparrow\infty$.
\end{proposition}

\begin{proposition}\label{diff0}
 For $p>2$ and  for all  $\b<\b_p$, 
\be\label{Tayf2}
\lim_{N \uparrow\infty}\alpha_N(p)\left|\ZZ_N(\b)- T_N(\b) \right|=0, \text{ in probability}.
 \ee
\end{proposition}
\begin{remark} 
In view of the fact that by Lemma \thv(veryclose.1) $\ZZ_N(\b)$ and $Z^\leq_\e$ differ only by 
an exponentially small quantity, Proposition \thv(diff0) is immediate if we show that 
\be\label{Tayf2-1}
\lim_{N \uparrow\infty}\alpha_N(p)\left|Z^\leq_\e- T_N(\b) \right|=0, \text{ in probability}.
 \ee
\end{remark}

The proof of these two claims is given in the next subsections. Before that, 
we  emphasise that the different limiting pictures depending on the parity of $p> 2$ stem, in fact, from the $T_N$-term: \begin{itemize}
\item $p$ odd. In this case $\E_{\s}\left(H_N(\s)^3\right)=0$ by antisymmetry (see \eqref{esperance0} below), in which case 
\be\label{b4}
T_N(\b) = 1-\b^4\frac{\E_{\sigma} \left(H_N(\s)^2\right)^2}{8}+\b^4\frac{\E_{\sigma} \left(H_N(\s)^4\right)}{4!}.
\ee
This should be contrasted to 
\item $p$ even. We will see in the course of the proof that the only  relevant term is, as a matter of fact,  
the third moment, with the second and fourth moments contributing nothing due to a "wrong" blow-up. In other words, it will become clear that 
\be\label{b3}
T_N(\b) = 1+\b^3\frac{\E_{\s}\left(-H_N(\s)^3\right)}{3!} + \text{"vanishing corrections"}.
\ee
\end{itemize}

We prove Propositions \thv(normal1) and \thv(diff0) in the remainder of this paper. 
As a first step, in Section \ref{expli_repr} below we provide some explicit formulas for the moments of $\E_\s H^k, k=2,3, 4$ which appear in the definition of $T_N(\b)$. Proposition \ref{normal1} for odd $p$ is then proven in Section \ref{podd} below, whereas the case of $p$ even in Section \ref{peven}; the proof of Proposition \ref{diff0} for even $p$ is given in Section \ref{pevenprim} and the proof for the odd $p$ case is finally given in Section \ref{poddprim}.

\subsection{Explicit representations of quenched moments} \label{expli_repr}
In the sequel we use the following abbreviation when summing over multi-indices $A,B\in I_N$. 
 \be
\sum_{(\neq)} J_A J_B \E_\s(\s_A \s_B) \equiv \sum_{A, B \in I_N: A\neq B} J_A J_B \E_\s(\s_A \s_B), 
\ee 
and similarly for sums involving a higher number of multi-indices, in which case we mean that all multi-indices
involved must be different.

For the different terms appearing in $T_N(\b)$, taking into account cancellations due to the averages over $\s$, we have the following representations.

\begin{lemma} \label{H4} We have
\be\label{H3} 
{\E_{\sigma} \left(-H_N(\s)^3\right)}={ a_N^3}\sum_{A,B,C \in I_N}J_{A}J_{B}J_{C} \E_{\sigma}\left(\s_{A}\s_{B}\s_{C}\right)={ a_N^3}\sum_{(\neq)} J_{A}J_{B}J_{C} \E_{\sigma}\left(\s_{A}\s_{B}\s_{C}\right)\,.
\ee
and
\be
-\frac18{\E_{\sigma} \left(H_N(\s)^2\right)^2}+\frac 1{4!}{\E_{\sigma} \left(H_N(\s)^4\right)}=-\frac{ a_N^4}{12}\sum_{A \in I}J_{A}^4+\HH_4,
\ee
where
\be
{\mathcal H}_4 \equiv \frac{a_N^4}{4!}\sum_{(\neq)} J_{A}J_{B}J_{C}J_{D} \E_{\sigma}\left(\s_{A}\s_{B}\s_{C}\s_{D}\right).
\ee

\end{lemma}
\begin{proof} 
Eq. \eqv(H3) is straightforward.
An elementary computations shows that
\be\label{esperance}
-{\E_{\sigma} \left(H_N(\s)^2\right)^2}=-{ a_N^4}\sum_{A, B \in I_N}J_{A}^2 J_{B}^2 = -{a_N^4}\sum_{(\neq)}J_{A}^2 J_{B}^2-{ a_N^4}\sum_{A \in I_N}J_{A}^4.
\ee
The fourth moment gives
\be
{\E_{\sigma} \left(H_N(\s)^4\right)}={a_N^4}\sum_{A, B, C, D \in I_N} J_{A}J_{B}J_{C}J_{D} \E_{\sigma}\left(\s_{A}\s_{B}\s_{C}\s_{D}\right). 
\ee
We now rearrange the summation according to the possible sub-cases: {\it i)} four multi-indices come in two distinct pairs (say $A=B$ and $C=D$ but $A\neq C$): in this case $\E_\s \s_A \s_B \s_C \s_D = \E_\s \s_A^2 \s_C^2 = 1$; {\it ii)} all four multi-indices coincide, in which case $\E_\s \s_A \s_B \s_C \s_D = \E_\s \s_A^4=1$; {\it iii)} at least one multi-index is different from all the others.
In this  case  the only non-vanishing contribution comes if four multi-indices are different. Hence
\be \label{esperance1}
{\E_{\sigma} \left(H_N(\s)^4\right)} = {3  a_N^4}\sum_{( \neq) } J_{A}^2J_{C}^2 +{a_N^4}\sum_{A \in I_N} J_{A}^4 +{ a_N^4}\sum_{(\neq)} J_{A}J_{B}J_{C}J_{D} \E_{\sigma}\left(\s_{A}\s_{B}\s_{C}\s_{D}\right),
\ee
where for the  first term on the right we use that there are $\binom{4}{2}=3$ ways to choose the pairs. Combining \eqref{esperance} and \eqref{esperance1}  yields the claim of the lemma. 
\end{proof}

\subsection{Proof of Proposition \ref{normal1}:  $p$ odd.} \label{podd} We first observe that 
\be\label{esperance0}
\E_{\s}\left(-H_N(\s)^3\right)= a_N^3\sum_{A,B,C \in I_N}J_{A}J_{B}J_{C} \E_{\s}\left(\s_{A}\s_{B}\s_{C}\right) =0,
\ee
since   $\s_A\s_B \s_C$ is a product of an \emph{odd} number of spins, and hence its expectation vanishes. Combining Lemma \ref{H4} and \eqref{esperance0}, it follows that
\be
\alpha_N(p)\left(T_N(\b)-1\right)=N^{p-2}\left(-\frac{\b^4 a_N^4}{12}\sum_{A \in I}J_{A}^4+\b^4\mathcal{H}_4\right).
\ee
First note that
\be\label{espLLN}
N^{p-2}\left(-\frac{\b^4 a_N^4}{12}\sum_{A \in I}J_{A}^4\right)=-N^{p-2}\frac{\b^4 N^2}{12\binom{N}{p}}\frac{1}{\binom{N}{p}}\sum_{A \in I}J_{A}^4 \rightarrow
-\frac{\b^4 p!}{4}, \as,
\ee
as $N\uparrow\infty$ by the strong law of large numbers.
 It  remains to prove that $N^{p-2} \mathcal H_4$  converges to a Gaussian with mean zero and variance 
 $\s(\b,p)^2$. 
 This will be done by proving that the moments of 
$N^{p-2} \mathcal H_4$ converge to those of the Gaussian.  We break this up into a series of lemmata.

\begin{lemma} (Second moment / variance). \label{varsecmom} For any $\b \geq 0$ and any $p\geq 3$,
 \be
\lim_{N \to +\infty} \b^8\E\left(\left(N^{p-2}  \mathcal{H}_4\right)^2\right)= \sigma(\b,p)^2,
\ee
\end{lemma}

\begin{lemma} (Even moments). \label{evmom} For any $\b \geq 0$, and $p$ odd, and for 
all $k\in \N$,
\be
\lim_{N \to +\infty} \b^{8k}\\E\left(\left(N^{p-2}  \mathcal{H}_4\right)^{2k}\right)=\frac{(2k)!}{2^kk!}\sigma(\b,p)^{2k}.
\ee
\end{lemma}

\begin{lemma} (Vanishing of odd moments). \label{vanodd} For any $\b \geq 0$, $p$ odd and for all $k\in\N$, 
 \be
\lim_{N \to +\infty} \E\left(\left(N^{p-2}  \mathcal{H}_4\right)^{2k+1}\right)=0.
 \ee
\end{lemma}
The remainder of this subsection is devoted to the proofs of these lemmata, which combined imply Proposition \ref{normal1} for $p$ odd.

\begin{proof}[Proof of Lemma \ref{varsecmom}]  We have that
 \bea\label{gastro100}\nonumber
\E\left(\mathcal{H}_4^2\right)&=&\frac{a_N^8}{4!^2}\sum_{\substack{A,B,C,D \in I_N \\ (\neq)}} \sum_{\substack{E,F,G,H \in I_N \\ (\neq)}} \E \left(J_{A}J_{B} \dots J_{H}\right) \E_{\sigma}\left(\s_{A}\s_{B}\s_{C}\s_{D}\right)\E_{\sigma'}\left(\s_{E}'\s_{F}'\s_{G}'\s_{H}'\right)\\\nonumber
&=&4!\frac{ a_N^8}{4!^2}\sum_{(\neq)} \E \left(J_{A}^2J_{B}^2J_{C}^2J_{D}^2\right) \E_{\sigma}\left(\s_{A}\s_{B}\s_{C}\s_{D}\right)\E_{\sigma'}\left(\s_{A}'\s_{B}'\s_{C}'\s_{D}'\right)\\
&=&\frac{ a_N^8}{4!}\sum_{(\neq)}  \E_{\sigma,\sigma'}\left(\s_{A}\s_{B}\s_{C}\s_{D}\s_{A}'\s_{B}'\s_{C}'\s_{D}'\right).
\eea
Here we used that in order to get a non-vanishing contributions, all the multi-indices in the first sum must be paired with one in the second sum. The number of such pairings is $4!$.

Next we express  $\E\left(\mathcal{H}_4^2\right)$ as a function of the overlaps.
\bea\label{gastro99}\nonumber
\E\left(\mathcal{H}_4^2\right)&=&\frac{ a_N^8}{4!}\bigg[\sum_{\substack{A,B,C,D \in I_N}}  \E_{\sigma,\sigma'}\left(\s_{A}\s_{B}\s_{C}\s_{D}\s_{A}'\s_{B}'\s_{C}'\s_{D}'\right)-3\sum_{\substack{A,B \in I_N\\ (\neq)}} \E_{\sigma,\sigma'}\left(\s_{A}^2\s_{B}^2\s_{A}'^2\s_{B}'^2\right)\\
&&-\sum_{A\in I_N} \E_{\sigma,\sigma'}\left(\s_{A}^4\s_{A}'^4\right)\bigg]\\\nonumber
&=&\frac{ a_N^8}{4!}\left[\sum_{\substack{A,B,C,D \in I_N}}  \E_{\sigma,\sigma'}\left(\s_{A}\s_{B}\s_{C}\s_{D}\s_{A}'\s_{B}'\s_{C}'\s_{D}'\right)- 3\left({\binom{N}{p}}^2-{\binom{N}{p}}\right)- \binom{N}{p}\right],
\eea
and therefore
\bea\label{gastro98}\nonumber
\E\left(\mathcal{H}_4^2\right)&=&\frac{a_N^8}{4!} \E_{\s,\s'}\left(\left(\sum_{A \in I_N} \s_{A}\s_{A}' \right)^4 \right)-\frac{3  a_N^8}{4!} {\binom{N}{p}}^2+\frac{2\b^8 a_N^8}{4!} \binom{N}{p}\\
&=&\frac{1}{4!}\sum_{m \in \Gamma_N} \left( N f_N^p\left( m \right)\right)^4p_N(m)-\frac{N^4}{8\binom{N}{p}^2}+O(N^{4-3p}),
 \eea
where we used \eqref{switch2}. Collecting the leading terms, we see that 
\be\label{gastrofin}
\E\left(\left(N^{p-2}  \mathcal{H}_4\right)^2\right)=\frac{1}{4!}\sum_{m \in \Gamma_N} \left( N^{\frac{p}{2}} f_N^p\left( m \right)\right)^4p_N(m)-\frac{ p!^2}{8}
+o(1).
 \ee
Furthermore, by \eqref{fondamental_zero}, we have that 
\be\label{gastrofin1}
N^{\frac{p}{2}}f_N^p\left(m\right)=\sum_{k=0}^{[p/2]} d_{p-2k}{\left(\sqrt{N}m\right)}^{p-2k}(1+O(1/N)),
\ee
and using this in the sum on the r.h.s. of \eqref{gastrofin} yields
\be\label{gastrofin2}
\E\left(\left(N^{p-2}  \mathcal{H}_4\right)^2\right)   =\frac{1}{4!}\sum_{m \in \Gamma_N} \left( \sum_{k=0}^{[p/2]} d_{p-2k}{\left(\sqrt{N}m\right)}^{p-2k}\right)^4p_N(m) \left(1+O\left(\frac{1}{N}\right)\right)-\frac{ p!^2}{8}+o_N(1)\,.
\ee
By Taylor-expanding in $m=0$, it can be checked that 
\be\label{Tay1}
p_N(m) = \frac{2}{\sqrt{2\pi}} \eee^{- N m^2/2}[1+o_N(1)].
\ee

It follows that the sum in \eqv(gastrofin2) converges to an 
integral, namely,
\bea \Eq(pregauss.1)
\lim_{N \uparrow \infty} \E\left(\left(N^{p-2}  \mathcal{H}_4\right)^2\right)  &=&  \frac{1}{12\sqrt{2 \pi}}\int_{-\infty}^{+\infty} \left(\sum_{k=0}^{[p/2]} d_{p-2k}{m}^{p-2k}\right)^4 \eee^{-\frac{m^2}{2}}dm-\frac{ p!^2}{8}\nonumber\\
&=&\b^{-8}\s(\b,p)^2.
\eea
This proves the lemma.
\end{proof}

\begin{proof}[Proof of Lemma \ref{evmom}] 
The $2k$-th moments of $\HH_4$ can be written as 
\be\label{combimpaire-}
\E\left(  \mathcal{H}_4^{2k}\right)=  \frac{  a_N^{8k}}{4!^{2k}}\E\left(\left(\sum_{(\neq)} J_{A}J_{B}J_{C}J_{D} \E_{\sigma}\left(\s_{A}\s_{B}\s_{C}\s_{D}\right)\right)^{2k}\right).
 \ee
and
 \bea\label{combimpaire3-}\nonumber
&&\E\left(\left( \sum_{\substack{A,B,C,D \in I_N \\ (\neq)}} J_{A}J_{B}J_{C}J_{D}\E_{\sigma}\left( \s_{A}\s_{B}\s_{C}\s_{D}\right)\right)^{2k}\right)= \\
&& \prod_{i=1}^{2k}\sum_{\substack{A_i,B_i,C_i,D_i \in I_N \\ (\neq)}}
 \E\left(\prod_{i=1}^{2k} J_{A_i}J_{B_i}J_{C_i}J_{D_i} \right) \prod_{i=1}^{2k} \E_{\sigma}\left( \s_{A_i}\s_{B_i} \s_{C_i} \s_{D_i}\right). 
 \eea
 Since the averages of odd powers of the 
 random variables $J$ vanish, only terms in the sums over
 the multi-indices in \eqv(combimpaire3-) give a non-zero 
 contribution where each multi-index occurs at least twice.
 Moreover, the leading order contribution
 comes from terms where each multi-index occurs 
 \emph{exactly} twice and where these pairings 
 take place between the multi-indices of two indices $i$ and $j$. We say a {\it pairing between the sums $i$ and $j$ takes place} as soon as $(A_i,B_i,C_i,D_i)=(\pi[A_j],\pi[B_j],\pi[C_j],\pi[D_j]))$ where $\pi$ is any permutation \footnote{note that we have $4!$ possible permutations.} on $(A_j,B_j,C_j,D_j)$.
 Since there are  $\frac{(2k)!}{k!2^k}$ different ways to construct such sum-pairings, we re-write the right-hand side of \eqv(combimpaire3-) as
\be\label{hard3'}\nonumber
\frac{4!^{k}(2k)!}{k!2^k}\sum_{
 (\neq)} \prod_{i=1}^{k}  \E\left(J_{A_i}^2 J_{B_i}^2  J_{C_i}^2 J_{D_i}^2\right)\left( \E_{\sigma}
\left( \s_{A_i}\s_{B_i} \s_{C_i} \s_{D_i}
\right)\right)^2+R_N(2k)
\equiv P_N(2k)+R_N(2k).
 \ee
The first term can be written as
\be\label{hard4'}
P_N(2k)= \frac{4!^k(2k)!}{k!2^k}\sum_{
(\neq)} 
\prod_{i=1}^k\left(\E_{\sigma}\left( \s_{A_i}\s_{B_i}\s_{C_i}\s_{D_i}\right)\right)^2.
 \ee 
 This term will converge to the appropriate moment of the
 Gaussian, whereas the  
 $R_N$-term tend to zero.

\begin{lemma}\TH(gauss.1)
With the notation above, 
\be
\Eq(gauss.2)
\lim_{N\uparrow\infty}\frac{N^{\left(2pk-4k\right)} a_N^{8k}\b^{8k}}{4!^{2k} }  P_N(2k) = \frac{(2k)!}{k!2^k}
\s(\b,p)^{2k}.
\ee
\end{lemma}
\begin{proof}
It is elementary to see that 
\be
\Eq(gauss.3)
\sum_{
 (\neq)}  \prod_{i=1}^{k} \E_{\sigma}
\left( \s_{A_i}\s_{B_i} \s_{C_i} \s_{D_i}
\right)^2=
\left(
\sum_{
 (\neq)} \left(\E_{\sigma}\left( \s_{A}\s_{B}\s_{C}\s_{D} \right)\right)^2\right)^k\left(1+O(N^{-p})\right).   
 \ee
 Recalling \eqv(gastro100),  
 \be
 \Eq(gauss.4)
\sum_{ (\neq)} \left(\E_{\sigma}\left( \s_{A}\s_{B}\s_{C}\s_{D} \right)\right)^2
 =\frac{4!}{a_N^8} \E\left(\HH_4^2\right).
 \ee
Putting these observations together and using \eqv(pregauss.1), we arrive at the assertion of the lemma.
\end{proof}
We now turn to the remainder term.

\begin{lemma}
\TH(nix.1)
\be
\Eq(nix.2)
\lim_{N\uparrow\infty} \frac{N^{\left(2pk-4k\right)} a_N^{8k}}{4!^{2k} }R_N(2k)=0.
\ee
\end{lemma}

\begin{proof}
Recall that the sums in \eqv(combimpaire3-) run over $8k$ multi-indices which by the 
pairing condition due to the $J$ is reduced to $4k$ multi-indices. In $P_N(2k)$, there are indeed that many
sums. We must show that in what is left, i.e. if pairings occur  that involve more than 
two groups, the effective number os summations is further reduced. 
This means that there are  terms where (double) products  of the following type appear:
\begin{enumerate}
\item
\[
 \E_{\s}\left(\s_A\s_B\s_C\s_D\right)\E_{\s}\left(\s_A\s_E\s_F\s_G\right),
\]
where $(E,F,G)$ do not coincide with any of the multi-indices $(A,B,C,D)$ or
\item
\[
\E_{\s}\left(\s_A\s_B\s_C\s_D\right)\E_{\s}\left(\s_A\s_B\s_E\s_F\right),
\]
where $(E,F)$ do not coincide with any of the multi-indices $(A,B,C,D)$ \footnote{Note that $\E_{\s}\left(\s_A\s_B\s_C\s_D\right)\E_{\s}\left(\s_A\s_B\s_C\s_E\right)$ implies that $E=D$ and is thus not a particular case.} or 
\item
\[
 \text{ \it sums which appear in pairs but at least one of the pairs coincide.}
\]
\end{enumerate}
The last case it trivially of lower order. 

We first look at the terms of type (1). 
They are of the form 
\be
\Eq(type.1)
\widetilde \sum_{(1)}  \E_{\s}\left(\s_A\s_B\s_C\s_D\right)\E_{\s}\left(\s_A\s_E\s_F\s_G\right) \prod_{i=1}^{2k-2} \E_{\sigma}\left( \s_{A_i}\s_{B_i} \s_{C_i} \s_{D_i}\right),
\ee
where the sum is over at most $4k$ different multi-indices where moreover $A,B,C,D,E,F,G$ respect the condition stated under (1) and of course the multi-indices 
with same index $i$ are all different. 
We first note that
\be
 \sum_{\substack{A,B,C,D \in I_N \\ (\neq)}}\E_{\s}\left(\s_A\s_B\s_C\s_D\right) \lesssim N^{2p}, 
\ee
since the expectation over $\s$ vanishes unless all 
$\s_i$ appearing in the product come in pairs. 
Thus, we may run $A$ over all $N^p$ values. Then $B,C,D$ may 
each match $k_B,k_C$ and $k_D$ with $k_B+k_C+k_D=p$ of the indices of $A$. 
Further, $C$ may in addition match $\ell_C$ of the $p-k_B$ free indices of $B$. 
Then $D$ must match the remaining $p-k_B-k_C$ unmatched indices of $A$, 
the $p-k_B-\ell_C$ unmatched indices of $B$ and the $p-k_C-\ell_C$ 
free indices of $C$. This leaves $N^{p-k_B}$ choices for $B$, $N^{p-k_C-\ell_C}$ choices for 
$C$, and just one for $D$.  Clearly, $\ell_C=k_D$, since $D$ must match the 
$p-k_B-k_C$
unmatched indices of $A$. Thus, the number of choices for the four multi-indices is
$N^{p+p-k_B+p-k_C-\ell_C}=N^{2p}$.   
If in addition one of the multi-indices is fixed, we are left with 
\be
 \sum_{\substack{B,C,D \in I_N \\ (\neq)}}\E_{\s}\left(\s_A\s_B\s_C\s_D\right) \lesssim N^{p}, 
\ee
where the $B,C,D$ must also be different from $A$. 
If two multi-indices are fixed,
\be
 \sum_{\substack{C,D \in I_N \\ (\neq)}}\E_{\s}\left(\s_A\s_B\s_C\s_D\right) \lesssim N^{p-1}.
\ee
This bound comes from the case when $B$ matches the largest  possible number of the
indices in $A$, namely $N-1$. In that case, $C$ has to just match the one remaining index from $A$, leaving $N^{p-1}$ choices that then have to be matched by $D$.
Finally, if all four multi-indices are fixed there is only one contribution. We see that the cost
of fixing one multi-index ist at least $N^{-p/2}$ which is achieved only if
four are fixed in the same pack of four (which corresponds to the terms in $P_N(2k)$).

Let us now return to the sum \eqv(type.1),
\be\label{first.1}
\widetilde \sum_{(1)}  \E_{\s}\left(\s_A\s_B\s_C\s_D\right)\E_{\s}\left(\s_A\s_E\s_F\s_G\right) \prod_{i=1}^{2k-2} \E_{\sigma}\left( \s_{A_i}\s_{B_i} \s_{C_i} \s_{D_i}\right).
 \ee
 The sum over the seven multi-indices $A,B,C,D,E,F,G$ gives at most $N^{3p}$ terms:
 The sum over $A$ gives $N^p$, and then, according to the discussion above, 
 the $B,C,D$ and the $E,F,G$ $N^p$ each.  The remaining sum is over
 $4(2k-2)$ multi-indices, of which $6$ have to be matched to $B,C,D,E,F,G$, and all others must be paired. This leaves $4k-7$ sums over multi-indices to be summed,
 which gives due to the constraints created by the $\s$-sums at most $N^{p(2k-7/2)}$ 
 terms. So overall, \eqv(first.1) is bounded by
a constant times $N^{p(2k-1/2)}\ll N^{2kp}$.

Terms of Type (2) are of the form 
\be\label{second}
\widetilde \sum_{(2)}  \E_{\s}\left(\s_A\s_B\s_C\s_D\right)\E_{\s}\left(\s_A\s_B\s_E\s_F\right) \prod_{i=1}^{2k-2} \E_{\sigma}\left( \s_{A_i}\s_{B_i} \s_{C_i} \s_{D_i}\right).
 \ee
 To bound the sum over the first six multi-indices, we have to be more careful.
 First, there are $N^p$ choices for $A$. Then, if we choose $B$ such that $k_B$ indices 
 match those of $B$, there are $N^{p-k_B}$ choices for $B$. Finally, we must choose
 $k_C$ and $\ell_C$ as in the discussion above, thus that $k_B+k_C+\ell_C=p$, and
 equally $k_E$ and $\ell_E$ with the same property. This allows $N^{p-k_B}$ choices for each of these multi-indices. Finally, $E$ and $F$ are determined. Altogether, this leaves
 $N^{4p-k_B-k_C-\ell_C-k_E-\ell_E}=N^{2p+k_B}$ terms, for $k_B$ given. But since 
 $B\neq A$, $k_B\leq p-1$, so that 
 the sum over these $6$ indices contribute at most $O(N^{3p-1})$ terms. 
  From the remaining $4(2k-2)$ multi-indices, 
four are fixed to match $C,D,E,F$, and all others must be paired. 
This leaves $2(2k-3)$ free multi-indices which can at most contribute $N^{p(2k-3)}$ terms. 
So in all the sum in \eqv(second) is bounded by $Cont. N^{2kp-1}$, which is again 
of lower order than $N^{2kp}$.

Finally, if any multi-index occurs four times, we loose a factor of $N^{2p}$ and also 
these terms are negligable. Combining these observations we have proven the lemma.
\end{proof}
The assertion  of Lemma \thv(evmom) follows immediately.
\end{proof}

\begin{proof}[Proof of Lemma \ref{vanodd}] In the case of odd moments, pairing of the multi-indices between always just two blocks is obviously impossible, so that the
terms that contributed to the leading $P_N(2k+1)$ do not exist. 
 Thus 
 \be\label{RN1}
\E\left(\left(N^{p-2}  \mathcal{H}_4\right)^{2k+1}\right)=  \frac{N^{(p-2)(2k+1)} a_N^{4(2k+1)}}{4!^{2k+1} } R_N(2k+1) \lesssim  \frac{1}{N^{(2k+1)p}} R_N(2k+1).
 \ee
By the same arguments as in the proof of Lemma \thv(nix.1), $R_N(2k+1)$ is of smaller order than $N^{(2k+1)p}$ and hence the right-hand side of \eqv(RN1) tends to zero.
This proves Lemma \thv(vanodd).
\end{proof}
This also concludes the proof of Proposition \thv(normal1) for $p$ odd.

\subsection{Proof of Proposition \ref{normal1}:  $p$ even.} \label{peven}

Recall that for $p$ even, 
\be\label{russlandeutsch}
\alpha_N(p) \left(T_N(\b)-1\right)=\b^4 N^{\left(\frac{3p}{4}-\frac{3}{2}\right)}\left(\frac{-\E_{\sigma} \left(H_N(\s)^2\right)^2}{8}+\frac{\E_{\sigma} \left(H_N(\s)^4\right)}{4!}
\right)+\b^3N^{\left(\frac{3p}{4}-\frac{3}{2}\right)}\frac{\E_{\sigma} \left(-H_N(\s)^3\right)}{3!}.
\ee
We first show that only the last term is relevant.
\begin{lemma}\TH(auchnix.1)
\be\Eq(auchnix.2)
\lim_{N\uparrow\infty}N^{\left(\frac{3p}{4}-\frac{3}{2}\right)}\left(-\frac{\E_{\sigma} \left(H_N(\s)^2\right)^2}{8}+\frac{\E_{\sigma} \left(H_N(\s)^4\right)}{4!}
\right)=0.
\ee
\end{lemma}
\begin{proof}
By Lemma \ref{H4}, 
\be\label{chiant}
N^{\left(\frac{3p}{4}-\frac{3}{2}\right)}\left(-\frac{\E_{\sigma} \left(H_N(\s)^2\right)^2}{8}+\frac{\E_{\sigma} \left(H_N(\s)^4\right)}{4!}\right)=-N^{\left(\frac{3p}{4}-\frac{3}{2}\right)}\frac{ a_N^4}{12}\sum_{A \in I}J_{A}^4+N^{\left(\frac{3p}{4}-\frac{3}{2}\right)}\mathcal{H}_4.
\ee
By the law of large numbers (see \eqv(espLLN)), the first term in the right converges to zero in probability.  By Lemma \thv(varsecmom),
$N^{p-2}\HH_4$ converges to a constant in $L^2$. Since $\frac{3p}{4}-\frac{3}{2}<
p-2$ if $p>2$, this implies that the last term in \eqv(chiant) also converges to
zero in probability. This proves the lemma. 
\end{proof}
 Thus, it only  remains to prove that
\be
\b^3 N^{\left(\frac{3p}{4}-\frac{3}{2}\right)}\frac{\E_{\sigma} \left(-
H_N(\s)^3\right)}{3!}\limlaw\mathcal{N}(0,\sigma(\b,p)^2).
\ee
to conclude the proof of Proposition
\ref{normal1}. 
We break this up into three lemmata as in the odd case.

\begin{lemma} (Second moment). \label{varsecmom'} For any $\b \geq 0$, 
 \be
\lim_{N \to +\infty} \b^6\E\left(\left(N^{\left(\frac{3p}{4}-\frac{3}{2}\right)} \E_{\s}\left(\frac{-H_N(\s)^3}{3!}\right)\right)^2\right)= \sigma(\b,p)^2.
\ee
\end{lemma}
\begin{lemma} (Even moments). \label{evmom'} For any $\b \geq 0$, 
\be
\lim_{N \to +\infty} \b^{6k}\E\left(\left(N^{\left(\frac{3p}{4}-\frac{3}{2}\right)} \E_{\s}\left(\frac{-H_N(\s)^3}{3!}\right)\right)^{2k}\right)=\frac{(2k)!}{2^kk!}\sigma(\b,p)^{2k}.
\ee
\end{lemma}
\begin{lemma} (Odd moments). \label{vanodd'} For any $\b \geq 0$, 
 \be
\lim_{N \to +\infty} \E\left(\left(N^{\left(\frac{3p}{4}-\frac{3}{2}\right)} \E_{\s}\left(\frac{-H_N(\s)^3}{3!}\right)\right)^{2k+1}\right)=0.
 \ee
\end{lemma}

\begin{proof}[Proof of Lemma \ref{varsecmom'}]
We have that
\bea\label{Markov3}\nonumber
&&\E\left(\E_{\sigma} \left(H_N(\sigma)^3\right)^2\right)\\
&&={a_N^6}  \sum_{\substack{A,B,C \in I_N (\neq) \\ D,E,F  \in I_N (\neq) }} \E\left(J_AJ_BJ_CJ_DJ_EJ_F\right) \E_{\sigma,\sigma'}\left( \s_{A}\s_{B}\s_{C}\s_{D}'\s_{E}'\s_{F}'\right).
\eea
We rearrange the summation according to the possible sub-cases: {\it i)} all four multi-indices coincide, {\it ii)} four multi-indices coincide and two multi-indices come in a distinct pair;  {\it iii)} six multi-indices come in three different pairs. Thus the right-hand side of 
\eqv(Markov3) equals
\bea\label{Markov4}\nonumber
&&a_N^6\E\left(J^6\right) \sum_{A \in I_N} \E_{\sigma,\sigma'}\left(\s_{A}\s_{A}'\right)+a_N^6\binom{6}{2} \E\left(J^4\right) \E\left(J^2\right) \sum_{A \neq B \in I_N }\E_{\sigma,\sigma'}\left(\s_{A}\s_{B}'\right) \\
&&+6a_N^6\E\left(J^2\right)^3 \sum_{{A,B,C \in I_N  (\neq)  }}\E_{\sigma,\sigma'}
\left(\s_{A}\s_{B}\s_{C}\s_{A}'\s_{B}'\s_{C}'\right)\nonumber\\
&&=6{a_N^6} \sum_{A,B,C\in I_N  } \E_{\sigma,\sigma'}\left(\s_{A}\s_{B}\s_{C}
\s_{A}'\s_{B}'\s_{C}'\right),
\eea
where the factor $6$ accounts for the $3!$ possible pairings that all give the same contribution. In the last line we dropped the condition $(\neq)$, 
since all terms where this is not satisfied vanish.
 We conclude that
\be\label{Markov6}
\E\left({\E_{\sigma} \left(H_N(\sigma)^3\right)^2}\right)=6{ a_N^6}  \E_{\sigma,\sigma'}\left[\left( \sum_{A \in I_N}\s_{A}\s_{A}'\right)^3\right]=6 \sum_{m \in \Gamma_N} \left(Nf_N^p\left( m \right)\right)^3 p_N(m).
\ee
 From here we get
\be\label{Markovb}
 \E\left(N^{2\left(\frac{3p}{4}-\frac{3}{2}\right)}{\E_{\sigma} \left(H_N(\sigma)^3\right)^2}\right)={3!} \sum_{m \in \Gamma_N}\left(N^{\frac{p}{2}} f_N^p\left( m \right)\right)^3p_N(m).
\ee
Exactly as in the proof of Lemma \thv(varsecmom)
it now follows that
\be 
\lim_{N \uparrow \infty} \E\left(N^{2\left(\frac{3p}{4}-\frac{3}{2}\right)}\frac{\E_{\sigma} \left(H_N(\sigma)^3\right)^2}{3!^2}\right) = \frac{1}{3\sqrt{2\pi}}
\int_{-\infty}^{+\infty}\left(\sum_{k=0}^{[p/2]}  d_{p-2k}{m}^{p-2k}\right)^3\eee^{-\frac{m^2}{2}}dm,
\ee
which proves the lemma. 
\end{proof}

\begin{proof}[Proof of Lemma \ref{evmom'}] 
For $k>1$, we consider
 \be\label{hard}
\E\left(\left(N^{\left(\frac{3p}{4}-\frac{3}{2}\right)} \E_{\s}\left(\frac{H_N(\s)^3}{3!}\right)\right)^{2k}\right)=\frac{N^{\left(\frac{3pk}{2}-3k\right)} a_N^{6k}}{3!^{2k} }\E\left(\left( \sum_{(\neq)} J_AJ_BJ_C \E_{\sigma}\left( \s_{A}\s_{B}\s_{C}\right)\right)^{2k}\right).
 \ee 
Expanding the $2k$-moment inside the expectation yields
 \bea\label{hard2}\nonumber
&&\E\left(\left( \sum_{
(\neq)} J_{A_1}J_{B_1}J_{C_1}\E_{\sigma}\left( \s_{A_1}\s_{A_2}\s_{A_3}\right)\right)^{2k}\right)=\\
&& \prod_{i=1}^{2k} \sum_{
(\neq)} \E\left(\prod_{i=1}^{2k} J_{A_i} J_{B_i}  J_{C_i} \right) \prod_{i=1}^{2k} \E_{\sigma}\left( \s_{A_i}\s_{B_i}\s_{C_i}\right). 
 \eea
 We now proceed as in the case $p$ odd. The principal term in the sum comes form 
 the multi-indices within two blocks $i$, $j$ are 
 $(A_i,B_i,C_i)=(\pi[A_j],\pi[B_j],\pi[C_j]))$ matched. 
  Since there are  $\frac{(2k)!}{k!2^k}$ different ways to construct such sum-pairings, we re-write the right-hand side of \eqv(hard2) as
\bea\label{hard3}\nonumber
&&\frac{3!^k(2k)!}{k!2^k}\sum_{\substack{A_1,B_1,C_1 \dots A_{k},B_{k},C_{k} \in I_N \\ (\neq)}} \prod_{i=1}^{k}  \E\left(J_{A_i}^2 J_{B_i}^2  J_{C_i}^2\right)\prod_{i=1}^{k} \E_{\sigma}\left( \s_{A_i}\s_{B_i} \s_{C_i}\right)^2+R_N(2k) \\&&  \equiv P_N(2k)+R_N(2k).
 \eea
 
 As in the odd case, we have the following results.
 
\begin{lemma}\TH(gauss.101)
With the notation above, 
\be
\Eq(gauss.201)
\lim_{N\uparrow\infty}
N^{\left(\frac{3pk}{2}-3k\right)} \frac{\b^{6k} a_N^{6k}}{3!^{2k}}P_N(2k)= \frac{(2k)!}{k!2^k}\sigma(\b,p)^{2k}.
 \ee
\end{lemma}

\begin{lemma}
\TH(nix.101)
\be
\Eq(nix.201)
\lim_{N\uparrow\infty}
N^{\left(\frac{3pk}{2}-3k\right)} \frac{a_N^{6k}}{3!^{2k}}R_N(2k)=0.
\ee
\end{lemma}

\begin{proof}[Proof of Lemma \thv(gauss.101)]
The proof is completely analogous to that of Lemma \thv(gauss.1) and will  be omitted.
\end{proof}

\begin{proof}[Proof of Lemma \thv(nix.101)]
The non-trivial terms that appear in the expression for $R_N(2k)$ must contain a term of the form 
\be\Eq(first.101)
 \E_{\s}\left(\s_A\s_B\s_C\right)\E_{\s}\left(\s_A\s_D\s_E\right),
\ee
where $(D,E)$ do not coincide with any of the multi-indices $(A,B,C)$ \footnote{Note that $\E_{\s}\left(\s_A\s_B\s_C\right)\E_{\s}\left(\s_A\s_B\s_D\right)$ implies that $C=D$ and is thus not a particular case.} 
That is, we have to control sums of the form 
 \be\label{combimpaire31p}
\widetilde \sum_{(1)}  \E_{\s}\left(\s_A\s_B\s_C\right)\E_{\s}\left(\s_A\s_D\s_E\right) \prod_{i=1}^{2k-2} \E_{\sigma}\left( \s_{A_i}\s_{B_i} \s_{C_i}\right),
\ee
where $A,B,C,D,E$ are as above and all multi-indices must be paired.
By a computation analogous to that in the proof of Lemma \thv(gauss.1), we get that 
\be
 \sum_{\substack{A,B,C \in I_N \\ (\neq)}}\E_{\s}\left(\s_A\s_B\s_C\right) \lesssim N^{\frac{3p}{2}}.
\ee
Looking at \eqv(first.101), we see that the sum over the $(A,B,C,D,E)$
produces $O(N^{2p})$ terms. Of the remaining $3(2k-2)$ multi-indices, four must 
match $B,C,D,E$ while the remaining ones must be paired. This leaves
$(3k-5)$ free multi-indices to sum over. This yields at most $N^{p(3k-5)/2}$ terms, so that altogether the sum in 
\eqv(combimpaire31p) is of order at most $N^{p(3k-1)/2}$. Inserting this into \eqv(nix.201) shows that the left-hand side is of order $N^{-p/2}$ and converges to zero as claimed.
\end{proof}
Lemma \thv(gauss.101) and Lemma \thv(nix.101) yield the assertion of Lemma 
\thv(evmom').
\end{proof}

\begin{proof}[Proof of Lemma \ref{vanodd'}]
$\E_{\s}\left(H_N(\s)^3\right)^{2k+1}$ is a sum of a product of $6k+3$ standard normal random variables, which is an odd number: At least one of the $J.$ will be to the power of an odd number. The expectation value of $\E_{\s}\left(H_N(\s)^3\right)^{2k+1}$ with respect to $\E$ is thus equal to $0$.
\end{proof}

\subsection{Proof of Proposition \ref{diff0}: $p$ even.} \label{pevenprim}

We want to show that 
\be\Eq(Bonn)
\lim_{N\uparrow\infty} N^{3p/4-3/2} \left|\ZZ_N(\b)- T_N(\b)\right|=0.
\ee
Using the definition of $T_N(\b)$  
\bea\label{Mainz}\nonumber
 \left|\ZZ_N(\b)- T_N(\b)\right|&\leq&  \left|\ZZ_N(\b)-1- \frac{\b^3}{3!}{\E_{\sigma}\left(-H_N(\s)^3\right)}\right|
+ \left|\frac{\b^4}{8}{\E_{\sigma} \left(H_N(\s)^2\right)^2}-\frac{\b^4}{4!}{\E_{\sigma} \left(H_N(\s)^4\right)}\right|\\\nonumber
 &\leq& \left|Z_\e^\leq -\ZZ_N(\b)\right|+ \left|\E\left(Z_\e^\leq\right) -1\right|
 + \left|Z_\e^\leq-\E\left(Z_\e^\leq\right)-\frac{\b^3}{3!}{\E_{\sigma} \left(-H_N(\s)^3\right)}\right|\\&+&
 \left|\frac{\b^4}{8}{\E_{\sigma} \left(H_N(\s)^2\right)^2}-\frac{\b^4}{4!}{\E_{\sigma} \left(H_N(\s)^4\right)}\right|.
 \eea
 The first term in the second line is negligible by Lemma \eqv(veryclose.1), the second by \eqv(moment13) together with Lemma \eqv(veryclose.1).
 By Lemma \thv(auchnix.1), the last term on the right of \eqv(Mainz) will vanish if it is 
 inserted into  \eqv(Bonn).
 To control the remaining third  term, we bound its second moment,
\bea\label{Markov1}\nonumber
&&\E\left(  \left|Z_\e^\leq-\E(Z_\e^\leq)-\frac{\b^3}{3!}{\E_{\sigma} \left(-H_N(\s)^3\right)}\right|^2
\right)=\E\left( \left( Z_\e^\leq\right)^2-\left(\E(Z^\leq_\e)\right)^2\right)\\
&&+\frac {\b^6}{3!^2}\E\left({\E_{\sigma} \left(H_N(\s)^3\right)^2}\right)-\frac{2\b^3}{3!}\E\left( \E_{\sigma} \left(-H_N(\s)^3\right) \left( Z^{\leq}_\e-\E\left(Z_\e^\leq\right)\right)\right).
\eea
$\E\left( \E_{\sigma} \left(-H_N(\s)^3\right)\right)=0$ by symmetry. Therefore, the right-hand side of \eqv(Markov1) is equal to 
\be\label{Markov2}
\E\left((Z_\e^\leq)^2\right)- \left(\E\left(Z_\e^\leq\right)\right)^2
-\frac{\b^6}{3!^2}\E\left({\E_{\sigma} \left(H_N(\s)^3\right)^2}\right)-\frac{2\b^3}{3!}\E\left({\E_{\sigma} \left(-H_N(\s)^3\right)}\left({Z_\e^\leq}-\frac{\b^3}{3!}{\E_{\sigma} \left(-H_N(\s)^3\right)}\right)\right). 
\ee
In order to prove that the first term on the r.h.s. of \eqref{Mainz} vanishes, it thus remains to prove that
\begin{lemma}  \label{claim5} For all $\b<\b_p$, 
\be\label{claim5'}
\lim_{N \uparrow\infty}N^{\left(\frac{3p}{2}-3\right)}\left|\E\left((Z_\e^\leq)^2\right)- \left(\E\left(Z_\e^\leq\right)\right)^2
-\frac{\b^6}{3!^2}\E\left({\E_{\sigma} \left(H_N(\s)^3\right)^2}\right)\right|=0 
\ee
\end{lemma}
and
\begin{lemma} \label{claim6} For all $\b \in \R_+$, 
\be\label{claim6'}
\lim_{N \uparrow\infty}N^{\left(\frac{3p}{2}-3\right)}\left|\E\left({\E_{\sigma} \left(-H_N(\s)^3\right)}\left(Z_\e^\leq-\frac{\b^3}{3!}{\E_{\sigma} \left(-H_N(\s)^3\right)}\right)\right)\right|=0.
\ee
\end{lemma}
 Lemma \ref{claim5} and \ref{claim6} clearly imply Proposition \ref{diff0} for $p$ even.

 \begin{proof} [Proof of Lemma \ref{claim5}] 
We will now improve the estimate of the second moment of ${Z_\e^\leq}$
started with Eq. \eqv(mom21'). 
We write 
\be
\E\left[(Z_\e^\leq)^2\right]= A+B,
\ee
with $A, B$ are given in \eqv(mom21') and $A\leq A_1+A_2$, with $A_1,A_2$ defined 
in \eqv(mom22) and \eqv(mom22.1). The estimates obtained in Section 2 for $B$ (Lemma \thv(B.1)) and 
$A_1$ (Eq. \eqv(facile2)) are good enough, but we need to improve the bound on $A_2$.
 Recall that in the final bound \eqv(badbound.1) for $A_1$ there was an error term of order $N^{3-3p/2}$, which would not vanish if multiplied with the $N^{3p/2-3}$.
 This term is due to the cubic term in the expansion \eqv(strange). 
 But this term reads 
 \be\Eq(luck.3)\frac1{3!}
  \sum_{m \in \Gamma_N}{\left(\frac{\b^2 N f_N^p\left( m \right)} {2 \b^2 a_N^2+1}\right)^3} p_N\left(m\right).
\ee
 But recall that 
\be\label{Markov6b}
\E\left(\frac{\b^6}{3!^2}{\E_{\sigma} \left(H_N(\s)^3\right)^2}\right)= \sum_{m \in \Gamma_N} \frac{ \left(\b^2 Nf_N^p\left( m \right)\right)^3 }{3!}p_N(m).
\ee
Therefore, im the expression in \eqv(claim5'), this term exactly cancels the unpleasant cubic term in the expansion of $A_2$.

 Recalling \eqv(strange.2),
\be\label{finito2'''-}
A_2
 =
 \sum_{m \in \Gamma_N} \left( 1
 +\frac12{\left(\frac{\b^2 N f_N^p\left( m \right)} {2 \b^2 a_N^2+1}\right)^2}
 +\frac1{3!}{\left(\frac{\b^2 N f_N^p\left( m \right)} {2 \b^2 a_N^2+1}\right)^3}\right) p_N\left(m\right) + O\left(N^{4-2p}\right).
\ee
 Hence, we arrive at 
\bea\label{finito8}\nonumber
\E\left({Z_{\epsilon}^{\leq}}^2\right)&=&\left(1+\frac{\b^4Na_N^2}{2}+\frac{1}{3!}\sum_{m \in \Gamma_N} {\left(\b^2 N f_N^p\left( m \right)\right)^3} p_N(m) \right)\left(1-\b^4N a_N^2+O(N^{4-2p})\right)\\
&=&1-\frac{\b^4Na_N^2}{2}+\frac{1}{3!}\sum_{m \in \Gamma_N} {\left(\b^2 N f_N^p\left( m \right)\right)^3} p_N(m) +O(N^{4-2p}).
\eea
By \eqref{finito5}, 
\be\label{Markov7}
\E({Z_{\epsilon}^{\leq}})^2=1- \frac{\b^4 N a_N^2}{2}+O\left(N^{4-2p}\right),
\ee
finally using \eqv(Markov6b), we get Lemma
 \ref{claim5} and the lemma is proven.
\end{proof}

 \begin{proof} [Proof of Lemma \ref{claim6}] By definition of $Z_{\epsilon}^{\leq}$, we can re-write
\bea\label{Markov9}
&&\left|\E\left({\E_{\sigma} \left(-H_N(\s)^3\right)}\left({Z_{\epsilon}^{\leq}}-\frac{\b^3}{3!}{\E_{\sigma} \left(-H_N(\s)^3\right)}\right)\right)\right|\\\nonumber
&&=\left|\E\left({\E_{\sigma} \left(-H_N(\s)^3\right)}
\left(\mathcal{Z}_N(\b)-Z_{\epsilon}^{>}\right)\right)-\frac{\b^3}{3!}\E \left( \left({\E_{\sigma} \left(H_N(\s)^3\right)}\right)^2\right)\right|\\
&&\leq \left|\E\left({\E_{\sigma} \left(-H_N(\s)^3\right)}\mathcal{Z}_N(\b)\right)-\frac{\b^3}{3!}\E \left( \left({\E_{\sigma} \left(H_N(\s)^3\right)}\right)^2\right)\right|+\left| \E\left({\E_{\sigma} \left(H_N(\s)^3\right)}Z_{\epsilon}^{>}\right) \right|.\nonumber
\eea
The first term of the last line can be calculated explicitly. By \eqref{H3}, 
\bea\label{Markov10}\nonumber
&&\E\left({\E_{\sigma} \left(-H_N(\s)^3\right)}\mathcal{Z}_N(\b)\right)
={ a_N^3} \sum_{(\neq)}\E_{\sigma}\left(\s_{A}\s_{B}\s_{C}\right) \E\left( J_A J_B J_C \mathcal{Z}_N(\b)\right)\nonumber\\
&&\quad={ a_N^3} \sum_{(\neq)}\E_{\sigma}\left(\s_{A}\s_{B}\s_{C}\right) \E_{\sigma'} \E\left( J_A J_B J_C \eee^{-\b H_N(\s')-NJ_N(\b)}\right).
\eea
 Now,   
\bea\label{Markov13}
&&\E \left( J_A J_B J_C \eee^{-\b H_N(\s)-NJ_N(\b)}\right)=
\E \left( J_A J_B J_C  \eee^{\sum_{D\in I_N} \left(\b a_N\s_DJ_D-\frac{\b^2a_N^2}2J_D^2\right)}\right)\nonumber\\
&&\quad=\prod_{D \ \in I \setminus \{A,B,C\}}\E \left(\eee^{ \b a_N \s_{D}J_{D}- \frac{\b^2 a_N^2}{2}J_{D}^2}\right)\prod_{D \ \in \{A,B,C\}}\E \left(J_D \eee^{a_N \b \s_{D}J_{D}- \frac{a_N^2 \b^2}{2}J_{D}^2}\right).
\eea
We already have computed the terms in the first product, see \eqref{moment11'}.
For the second, we get by elementary integration,
\be\label{Markov14}
\E \left(J_{D}\eee^{ \b a_N \s_{D}J_{D}- \frac{1}{2} \b^2 a_N^2J_{D}^2}\right)
=\eee^{\left(\frac{\b^2 a_N^2}{2(1+2 \b^2 a_N^2)}-\frac{1}{2}\ln{(1+\b^2 a_N^2)}\right)}\frac{ \b a_N \s_{D}}{\left(1+ \b^2 a_N^2\right)}.
\ee
Therefore, 
\be\label{Markov15}
\E \left( J_A J_B J_C \eee^{H-NJ_N(\b)}\right)= \eee^{\binom{N}{p}\left(\frac{\b^2 a_N^2}{2(1+2 \b^2 a_N^2)}-\frac{1}{2}\ln{(1+\b^2 a_N^2)}\right)}\frac{ \b^3 a_N^3 \s_{A}\s_{B}\s_{C}}{\left(1+\b^2  a_N^2\right)^3}.
\ee
Using \eqref{Markov15} in  \eqref{Markov10} 
gives that 
\bea\label{Markov16}\nonumber
\E\left({\E_{\sigma} \left(-H_N(\s)^3\right)}\mathcal{Z}_N(\b)\right)&=&\frac{\b^3a_N^6}{\left(1+ \b^2 a_N^2\right)^3} \eee^{\binom{N}{p}\left(\frac{\b^2 a_N^2}{2(1+2 \b^2 a_N^2)}-\frac{1}{2}\ln{(1+\b^2 a_N^2)}\right)} \sum_{(\neq)}\E_{\sigma}\left(\s_{A}\s_{B}\s_{C}\right)^2  
\\
&= &\frac{\b^3\E \left( \left({\E_{\sigma} \left(-H_N(\s)^3\right)}\right)^2\right)}{3!\left(1+ \b^2 a_N^2\right)^3} \eee^{\binom{N}{p}\left(\frac{\b^2 a_N^2}{2(1+2 \b^2 a_N^2)}-\frac{1}{2}\ln{(1+\b^2 a_N^2)}\right)}.
\eea
Using \eqref{Markov16}, we get
\bea\label{Markov17}\nonumber
&&\left|\E\left({\E_{\sigma} \left(-H_N(\s)^3\right)}\mathcal{Z}_N(\b)\right)-\frac {\b^3}{3!}
\E \left( \left({\E_{\sigma} \left(-H_N(\s)^3\right)}\right)^2\right)\right| \\
&&=\frac {\b^3}{3!}\E \left( \left(\E_{\sigma} \left(-H_N(\s)^3\right)\right)^2\right) 
\left( \frac{\eee^{\binom{N}{p}\left(\frac{\b^2 a_N^2}{2(1+2 \b^2 a_N^2)}-\frac{1}{2}\ln{(1+
\b^2 a_N^2)}\right)}}{\left(1+\b^2  a_N^2\right)^3}-1\right). 
\eea
A simple expansion shows that 
\be\Eq(smallish)
\frac{\eee^{\binom{N}{p}\left(\frac{\b^2 a_N^2}{2(1+2 \b^2 a_N^2)}-\frac{1}{2}\ln{(1+
\b^2 a_N^2)}\right)}}{\left(1+\b^2  a_N^2\right)^3}-1=O(N^{2-p}).
\ee
Since 
\be\Eq(convergerei)
N^{\frac{3p}2-3}\frac {\b^6}{3!^2}\E \left( \left(\E_{\sigma} \left(-H_N(\s)^3\right)\right)^2\right) 
\rightarrow \sigma(\b, p)^2,
\ee
 and therefore
\be\label{Markov17bonus}
\lim_{N \to \infty }N^{\left(\frac{3p}{2}-3\right)}\left|\E\left({\E_{\sigma} \left(-H_N(\s)^3\right)}
\mathcal{Z}_N(\b)\right)-\frac{\b^3}{3!}\E \left( \left({\E_{\sigma} \left(-H_N(\s)^3\right)}\right)^2\right)\right|=0.
\ee
It  remains to prove that $\left| \E\left(\frac{\E_{\sigma} \left(-H_N(\s)^3\right)}{3!}Z_{\epsilon}^{>}\right) \right|$ tends to $0$. To see that, we use the H\"older inequality.
\bea\label{Markov18}\nonumber
&&\left| \E\left({\E_{\sigma} \left(-H_N(\s)^3\right)}Z_{\epsilon}^{>}\right) \right|
\leq \E\left(\left| {\E_{\sigma'} \left(-H_N(\s)^3\right)}\right| \E_\s\left(\eee^{-\b H_N(\s)} \1_{\{-H_N(\s)> (1+\epsilon) \b N\}}\eee^{-NJ_N(\b)}\right) \right) \\\nonumber
&&\quad=  \E_\s\left( \E\left(\left| {\E_{\sigma'} \left(-H_N(\s)^3\right)}\right|\eee^{-\b H_N(\s)} \1_{\{H_N(\s)> (1+\epsilon) \b^2 N\}}\eee^{-NJ_N(\b)}\right) \right) \\
&&\quad\leq \E_\s \bigg(\E \left(\eee^{ q_1 \b \sqrt N X_\s} \1_{\{X_\s>(1+\epsilon) \b \sqrt N\}}\right)^{\frac{1}{q_1}} \E  \left( \left| {\E_{\sigma'} \left(H_N(\s)^3\right)}\right|^{q_2} \eee^{-q_2NJ_N(\b)} \right)^{\frac{1}{q_2}}\bigg),
\eea
for  $\frac{1}{q_1}+\frac{1}{q_2}=1$. For the last factor, the Cauchy-Schwarz inequality gives

\be\label{Markov19}
\E  \left( \left| {\E_{\sigma'} \left(H_N(\s)^3\right)}\right|^{q_2} \eee^{-q_2NJ_N(\b)} \right)^{\frac{1}{q_2}}
\leq  \E  \left( \left| {\E_{\sigma'} \left(-H_N(\s)^3\right)}\right|^{2q_2} \right)^{\frac{1}{2q_2}}\E \left( \eee^{-2q_2NJ_N(\b)} \right)^{\frac{1}{2q_2}}.
\ee

Again by {\bf Fact I} in the appendix, the last line in 
\eqv(Markov18) is bounded from above by
\be\label{Markov20}
 \eee^{\left(-\frac{(1+\epsilon)^2 \b^2 N}{2 q_1}+(1+\epsilon) \b^2 N \right)} \left(\E  \left( \left| \E_{\sigma} \left(-H_N(\s)^3\right)\right|^{2q_2} \right)^{\frac{1}{2q_2}} \E \left( \eee^{-2q_2NJ_N(\b)} \right)^{\frac{1}{ 2q_2}}\right).
\ee
Finally, by explicit computation,
\be\label{Markov21}
\E \left( \eee^{-2q_2NJ_N(\b)} \right)^{1/ 2q_2}= \eee^{-N\b^2/2 + O\left( N^{2-p}\right)} .
\ee
Combining \eqref{Markov20} and \eqref{Markov21}, we obtain
\bea\label{Markov22}
&&\left| \E\left({\E_{\sigma'} \left(H_N(\s)^3\right)}Z_{\epsilon}^{>}\right) \right| 
\nonumber\\
&&\leq  \exp{\left(-\b^2 N \left( \frac{\epsilon^2}{2} + O\left(q_1-1\right)+O\left( N^{2-p} \right) \right)\right)} \E  \left( \left| {\E_{\sigma'} \left(-H_N(\s)^3\right)}\right|^{2q_2} \right)^{\frac{1}{2q_2}}.
\eea
For every $\epsilon>0$, we can choose $q_1$ close to $1$ such that the first term on the r.h.s. of \eqref{Markov22} is exponentially small. The second term will however stay polynomial. This concludes the proof of Lemma \ref{claim6}.
\end{proof} 

This concludes the proof of Proposition \ref{diff0} in case of $p$ even. \\
${}\hfill \square$

\subsection{Proof of Proposition \ref{diff0}: $p$ odd.} \label{poddprim}
The proof in the odd case is in principle similar to the even case. It is enough to show that 
\be\label{aie}
\lim_{N \uparrow\infty}N^{p-2}\left({Z_{\epsilon}^{\leq}}- T_N(\b)\right)=0,
 \ee
 in probability. Using \eqref{esperance0} we decompose
\be\label{aie1'}
\left|Z_{\epsilon}^{\leq}- T_N(\b)\right| \leq  \left|Z_{\epsilon}^{\leq}-\b^4\mathcal{H}_4-\E\left(Z_{\epsilon}^{\leq}\right)\right| +\left|\E\left(Z_{\epsilon}^{\leq}\right)-1+\frac{2\b^4 a_N^4}{4!}\sum_{A \in I}J_{A}^4\right|.
 \ee
 The second term is irrelevant. Using \eqv(moment13.1) and  the law of large numbers
 from  \eqref{espLLN}, we see that the second term is smaller than $o(N^{2-p})$ and hence gives 
 a vanishing contribution to \eqv(aie).
  For the first term in \eqv(aie1') we control its second moment. We write
\bea\label{aie2}\nonumber
\E\left( \left(Z_{\epsilon}^{\leq}-\b^4\mathcal{H}_4-\E\left(Z_{\epsilon}^{\leq}\right) \right)^2 \right)&=&2\b^4\E\left(\mathcal{H}_4\E\left(Z_{\epsilon}^{\leq}\right)\right)-2\b^4\E\left(\mathcal{H}_4 Z_{\epsilon}^{\leq}\right)+\E\left({Z_{\epsilon}^{\leq}}^2\right)-\E\left({Z_{\epsilon}^{\leq}}\right)^2\\
&&\quad+\b^8\E\left({\mathcal{H}_4}^2\right)\\\nonumber
&=&2\b^4\E\left( \mathcal{H}_4 (\b^4\HH_4-Z_{\e}^\leq))\right)+\E\left( {Z_{\epsilon}^{\leq}}^2\right)-\E\left(  Z_{\epsilon}^{\leq}\right)^2-\b^8E\left( \mathcal{H}_4^2\right),
 \eea
 where we used that $\HH_4$ has zero mean.

We will prove the following two lemmata.
\begin{lemma}\label{claim7} For all $\b$, 
\be\label{claim7'}
\lim_{N \to +\infty}N^{2p-4}\left|\E\left( \mathcal{H}_4 (Z_{\epsilon}^{\leq}-\b^4\mathcal{H}_4)\right)\right|=0.
\ee
\end{lemma}
and 
\begin{lemma}\label{claim8} For all $\b<\b_p$, 
\be\label{claim8'}
\lim_{N \to +\infty}N^{2p-4}\left|\E\left({Z_{\epsilon}^{\leq}}^2\right)-\E\left({Z_{\epsilon}^{\leq}}\right)^2-\b^8\E\left(\mathcal{H}_4^2\right)\right|=0.
\ee
\end{lemma}

We will first prove Lemma \ref{claim7} by following exactly the same strategy as for the case $p$ even.
\begin{proof}[Proof of Lemma \ref{claim7}] 
The proof of this lemma is very similar to that of Lemma \thv(claim6) and we omit many
 details.
As  in \eqref{Markov9}, we start with 
\be\label{gastro}
\left|\E\left( \mathcal{H}_4 (Z_{\epsilon}^{\leq}-\b^4\mathcal{H}_4)\right)\right|\leq \left|\E\left( \mathcal{H}_4 \mathcal{Z}_N(\b)\right)-\b^4\E\left(\mathcal{H}_4^2\right)\right|+ \left|\E\left( \mathcal{H}_4  Z_{\epsilon}^{>}\right)\right|.
\ee
For the first term on the r.h.s. of \eqref{gastro}, we have
\be\label{gastro1}
\E\left( \mathcal{H}_4 \mathcal{Z}_N(\b)\right)
=\frac{a_N^4}{4!}\sum_{(\neq)}  \E_{\sigma}\left(\s_{A}\s_{B}\s_{C}\s_{D}\right) \E_{\s}\E\left(J_{A}J_{B}J_{C}J_{D}\eee^{-H_N(s)-NJ_N(\b)}\right).
\ee
Following now the exact same steps as in the proof of \thv(claim6),
we arrive at the analog of \eqv(Markov16),
\be\label{gastro5}
\E\left( \mathcal{H}_4 \mathcal{Z}_N(\b)\right)
= \frac{\b^4\E \left( \mathcal{H}_4^2\right)}{\left(1+ \b^2 a_N^2\right)^4} \eee^{\binom{N}{p}\left(\frac{\b^2 a_N^2}{2(1+2 \b^2 a_N^2)}-\frac{1}{2}\ln{(1+\b^2 a_N^2)}\right)}.
\ee
 From here one concludes that
\be\label{gastro6}
\lim_{N\uparrow\infty}N^{2p-4}\left|\E\left( \mathcal{H}_4 \mathcal{Z}_N(\b)\right)-\b^4\E\left(\mathcal{H}_4^2\right)\right|=0.
\ee
The second term on the right of \eqv(gastro) is shown to be exponentially small exactly as
the second term in \eqv(Markov9). This concludes the proof of Lemma \thv(claim7).
\end{proof}

\begin{proof}[Proof of Lemma \ref{claim8}]
It remains to prove that
\be
 \lim_{N \to +\infty}N^{2p-4}\left|\E\left({Z_{\epsilon}^{\leq}}^2\right)-\E\left({Z_{\epsilon}^{\leq}}\right)^2-\b^8\E\left(\mathcal{H}_4^2\right)\right|=0 .
\ee 
As in the proof of Lemma \thv(claim5), we improve the estimate on $\E\left((Z_\e^\leq)^2\right)$ by retaining an additional term in the expansion of the 
exponential that then is cancelled by the $\E\left(\HH_4^2\right)$. Again this involves only the term $A_2$. This time, this requires to push the expansion  further and to use
that
 \be\Eq(order.6)
\left|\exp( \xi) - 1 - \xi  - \frac{1}{2}\xi^2- \frac{1}{3!}\xi^3 - \frac{1}{4!}\xi^4-\frac{1}{5!}\xi^5\right|
\leq \frac{1}{6!}\xi^6 \exp |\xi|.  
\ee 
This leads to the estimate
\bea\label{gastro8}
&&\E\left({Z_{\epsilon}^{\leq}}^2\right)\\\nonumber
&&=\sum_{m \in \Gamma_N} \left( 1
+\frac 12\left(\frac{\b^2 N f_N^p\left( m \right)}{2a_N^2+1}\right)^2
+\frac1{4!}\left(\frac{\b^2 N f_N^p\left( m \right)} {2a_N^2+1}\right)^4\right)
 \eee^{\left( -
{\b^4Na_N^2}+O\left( N^{3-2p}\right)\right)}+o(N^{4-2p}),
\eea
where we used that the terms of odd order  vanish by symmetry. 
The quadratic term equals $\frac{\b^4N^2}{2\binom{N}{p}(2 \b^2 a_N^2+1)^2}$.
Moreover, the quartic term gives
\be
\frac1{4!}\sum_{m\in\G_N}p_N(m)\left(\frac{\b^2 N f_N^p\left( m \right)} {2a_N^2+1}\right)^4
=\b^8\E\left(\HH_4^2\right) +\frac{\b^8 N^2a_N^4}{8} +O\left(N^{4-3p}\right).
\ee
Furthermore, using \eqref{firstmomz} we have that
\bea\label{finito5--}
\E({Z_{\epsilon}^{\leq}})^2&=&\left(1- \frac{\b^4}{4} N a_N^2 +\frac{\b^8}{32} N^2 a_N^4 +O\left(N^{3-2p}\right)\right)^2\nonumber\\
&=&1- \frac{\b^4 Na_N^2}2+\frac{2\b^8 N^2a_N^4}{16}+O\left(N^{3-2p}\right).
\eea
Combining these observations, the assertion of Lemma \thv(claim8) follows.
\end{proof}

This concludes the proof of Proposition \thv(diff0) and hence of Theorem  \thv(fluctuationspin0).

\section{\bf Appendix}

We state three useful results for the convenience of the reader. The first concerns standard estimates for truncated exponential moments of Gaussian random variables.

\noindent{\bf Fact I.}\label{I}  {\it 
 Let $\xi$ be a Gaussian random variable with 
  $\E(\xi)=0$,  $\E(\xi^2)=1$. Then for all $a, b>0$
\be
\E[\eee^{a \xi} {\1}_{\{\xi>b\}}] \leq \frac{1}{\sqrt{2\pi
(b-a)}}\eee^{-b^2/2+ab}, \quad \,\text{if }b>a,
\ee
\be
\E[\eee^{a \xi} {\1}_{\{\xi<b\}}] \leq \frac{1}{\sqrt{2\pi
(a-b)}}\eee^{-b^2/2+ab}, \quad \,\text{if }b<a.
\ee
}
\\
The second is the Gaussian concentration of measure inequality, to be found, for example, in \cite{LT}.

\noindent{\bf Fact II.\ }\label{II.}{\it Assume that $f(x_1,\ldots,x_d)$ is a function on ${\R}^d$ 
 with a Lipschitz constant $L$. Let $J_1,\ldots, J_d$ be 
 independent standard Gaussian random variables. 
Then for any 
 $u>0$ 
\be
\P\{|f(J_1,\ldots, J_d)-\E \left( f(J_1,\ldots, J_d)\right )|>u \}\leq 2\exp\{-u^2/(2L^2)\}.       
\ee
}

\end{document}